\documentclass[final]{svjour3}                     
\pdfoutput=1
\smartqed

\usepackage{floatrow}

\usepackage[printonlyused,nolist]{acronym}
\usepackage{algorithm}
\usepackage{algpseudocode}
\usepackage{amsmath}
\usepackage{amssymb}
\usepackage{amstext}
\usepackage{array}
\usepackage{bbm}
\usepackage{bm}
\usepackage{commath}
\usepackage{enumitem}
\usepackage{geometry}
\usepackage{graphicx}
\usepackage{listings}
\usepackage{mathrsfs}
\usepackage{mdframed}
\usepackage{rotating}
\usepackage{subfig}
\usepackage{wasysym}

\graphicspath{{images/}}

\newcommand{\new}{} 

\newcommand{\mockalph}[1]{}

\newcommand{\Cov}{\mathbb{C}\mathrm{ov}}				
\newcommand{\sE}{\mathbb{E}}					
\newcommand{\Hfil}{\mathcal{H}}              

\makeatletter
\newcommand{\distas}[1]{\mathbin{\overset{#1}{\kern\z@\sim}}}%
\newsavebox{\mybox}\newsavebox{\mysim}
\newcommand{\distras}[1]{%
  \savebox{\mybox}{\hbox{$\scriptstyle#1$}}%
  \savebox{\mysim}{\hbox{$\sim$}}%
  \mathbin{\overset{#1}{\kern\z@\resizebox{\wd\mybox}{\ht\mysim}{$\sim$}}}%
}
\makeatother

\newcommand{\iid}{\distras{\mathrm{i.i.d.}}}
\newcommand{\getsIid}{\overset{\mathrm{i.i.d.}}{\longleftarrow}}
\newcommand{\Prob}{\mathbb{P}}				

\renewcommand{\epsilon}{\varepsilon}

\renewcommand{\hat}{\widehat}

\newcommand{\R}{\mathbb{R}}					
\newcommand{\Nat}{\mathbb{N}}					

\newcommand{\1}{\mathbbm{1}}				
\newcommand{\fn}{(\cdot)}					  
\newcommand{\Laplace}[1]{\mathscr{L}\left\{ #1 \right\}}
\newcommand{\Linv}[1]{\mathscr{L}^{-1}\left\{ #1 \right\}}
\newcommand{\e}{\mathrm{e}}         
\newcommand{\Oh}{\mathcal{O}}        

\newcommand{\cInt}[2][]{\lambda^*_{#1}(#2)} 
\newcommand{\bInt}{\lambda}								
\newcommand{\exite}{\mu}									
\newcommand{\excite}{\exite}									
\newcommand{\lm}{\overline{\lambda^*}}
\newcommand{\Ns}{\boldsymbol{N}}

\newcommand{\bft}{\boldsymbol{t}}
\newcommand{\T}{\boldsymbol{T}}

\newcommand{\thetas}{\boldsymbol{\theta}}

\newcommand{\alg}[1]   {Algorithm~\ref{#1}}   
\newcommand{\eq}[1]{\eqref{#1}}

\newcommand{\dfn}[1]  {Definition~\ref{#1}}	
\newcommand{\thrm}[1] {Theorem~\ref{#1}}	
\newcommand{\fig}[1]  {Fig.\,\ref{#1}}		
\newcommand{\secn}[1] {Section~\ref{#1}}	
\newcommand{\ssec}[1] {Subsection~\ref{#1}}	

\begin{acronym}
	\acro{a.s.}{almost surely}
	\acro{c.d.f.}{cumulative distribution function}
	\acro{i.i.d.}{independent and identically distributed}
	\acro{p.d.f.}{probability density function}
	\acro{w.p.}{with probability}
	\acro{ASX}{Australian stock exchange}
	\acro{ETF}{exchange Traded Fund}
	\acro{GMM}{generalised method of moments}
	\acro{HP}{Hawkes process}
	\acro{MLE}{maximum-likelihood estimator}
	\acro{SDE}{stochastic differential equation}
	\acro{Q--Q}{quantile--quantile}
\end{acronym}

\acused{i.i.d.}
\acused{c.d.f.}
\acused{p.d.f.}

\begin{document}

\title{Hawkes Processes}

\titlerunning{Hawkes Processes}

\author{Patrick J.\ Laub        \and
        Thomas Taimre           \and
        Philip K.\ Pollett
}

\institute{P.\ J.\ Laub \at
           Department of Mathematics, The University of Queensland, Qld 4072, Australia, and \\
           Department of Mathematics, Aarhus University, Ny Munkegade, DK-8000 Aarhus C, Denmark. \\
           \email{p.laub@[uq.edu.au\textbar math.au.dk]}
           \and
           T.\ Taimre \at
           Department of Mathematics, The University of Queensland, Qld 4072, Australia\\
           \email{t.taimre@uq.edu.au}
           \and
           P.\ K.\ Pollett \at
           Department of Mathematics, The University of Queensland, Qld 4072, Australia\\
           \email{pkp@maths.uq.edu.au}
}



\date{Last edited: \today}

\maketitle

\begin{abstract}
Hawkes processes are a particularly interesting class of stochastic
process that have been applied in diverse areas, from earthquake
modelling to financial analysis. They are point processes whose
defining characteristic is that they `self-excite', meaning that each
arrival increases the rate of future arrivals for some period of time.
Hawkes processes are well established, particularly within the
financial literature, yet many of the treatments are
inaccessible to one not acquainted with the topic. This survey
provides background, introduces the field and historical
developments, and touches upon all major aspects of Hawkes processes.
\end{abstract}

\section{Introduction} \label{intro}

Events that are observed in time frequently cluster natually. An
earthquake typically increases the geological tension in the region
where it occurs, and aftershocks likely follow~\cite{ogata1988}.
A fight between rival gangs can ignite a spate of criminal
retaliations~\cite{mohler2011}. Selling a significant quantity of a
stock could precipitate a trading flurry or, on a larger scale, the
collapse of a Wall Street investment bank could send shockwaves
through the world's financial centres~\cite{azizpour2010}. \new

The \acfi{HP}\acused{HP} is a mathematical model for these
`self-exciting' processes, named after its creator Alan G.
Hawkes~\cite{hawkes1971spectra}. The \ac{HP} is a counting process
that models a sequence of `arrivals' of some type over time, for
example, earthquakes, gang violence, trade orders, or bank defaults.
Each arrival \emph{excites} the process in the sense that the chance
of a subsequent arrival is increased for some time period after the
initial arrival. As such, it is a non-Markovian extension of the
Poisson process. \new

Some datasets, such as the number of companies defaulting on loans
each year~\cite{lando2010}, suggest that the underlying process is
indeed self exciting. Furthermore, using the basic Poisson process to model
say the arrival of trade orders of a stock is highly inappropriate,
because participants in equity markets exhibit a herding behaviour, a
standard example of economic reflexivity \cite{filimonov2012}. \new

The process of generating, model fitting, and testing the goodness of
fit of \acp{HP} is examined in this survey. As the \ac{HP} literature
in financial fields is particularly well developed, applications in
these areas are considered chiefly here. \new
\section{Background} \label{background}

Before discussing \acp{HP}, some key concepts must be elucidated.
Firstly, we briefly give definitions for counting processes and point
processes, thereby setting essential notation. Secondly, we discuss
the lesser-known conditional intensity function and compensator, both
core concepts for a clear understanding of \ac{HP}s.

\subsection{Counting and point processes}

We begin with the definition of a counting process. \new

\begin{definition}[Counting process] \label{counting_process}
A \emph{counting process} is a stochastic process $(N(t) : t \geq 0)$
taking values in $\Nat_0$ that satisfies $N(0) =
0$, is \ac{a.s.} finite, and is a right-continuous step function with
increments of size $+1$. Further, denote by $(\Hfil(u) : u \geq 0)$
the \emph{history} of the arrivals up to time $u$.
(Strictly speaking $\Hfil(\cdot)$ is a filtration, that is, an increasing
sequence of $\sigma$-algebras.)
\end{definition}

A counting process can be viewed as a cumulative count of the number
of `arrivals' into a system up to the current time. Another way to
characterise such a process is to consider the sequence of random
arrival times $\T = \{T_1, T_2, \dots \}$
at which the counting process $N\fn$ has jumped. The process defined
as these arrival times is called a point process, described in
\dfn{point_process} (adapted from \cite{carstensen2010}); see
\fig{typical_point} for an example point process and its associated
counting process. \new

\begin{definition}[Point process] \label{point_process}
If a sequence of random variables $\T = \{T_1, T_2, \dots \}$, taking
values in $[0,\infty)$, has $\Prob(0 \leq T_1 \leq T_2 \leq \dots) = 1$,
and the number of points in a bounded region is \ac{a.s.} finite, then
$\T$ is a \emph{(simple) point process}.
\end{definition}

\begin{figure}
  \centering
  \includegraphics[width=0.8\textwidth]{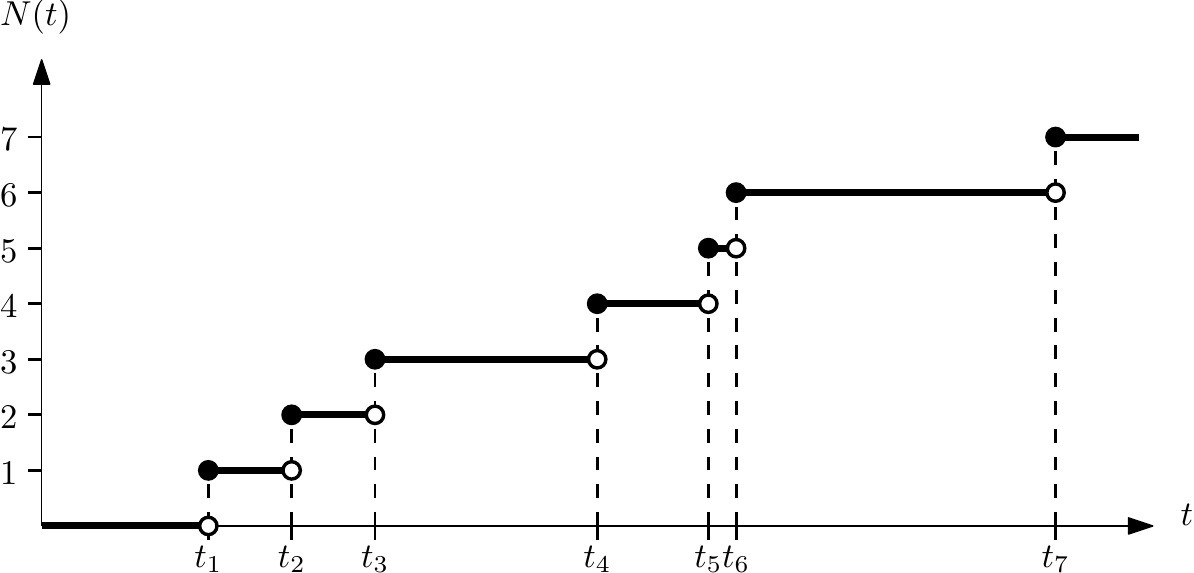}
  \caption[An example point process realisation]{An example point process realisation $\{t_1, t_2, \dots\}$ and corresponding counting process $N(t)$.}
  \label{typical_point}
\end{figure}

The counting and point process terminology is often interchangeable.
For example, if one refers to a Poisson process or a \ac{HP} then the
reader must infer from the context whether the counting process $N\fn$
or the point process of times $\T$ is being discussed. \new

One way to characterise a particular point process is to specify the
distribution function of the next arrival time conditional on the
past. Given the history up until the last arrival $u$, $\Hfil(u)$, define
(as per \cite{ozaki1979}) the conditional \ac{c.d.f.} (and
\ac{p.d.f.}) of the next arrival time $T_{k+1}$ as
\[
  F^*(t \,|\, \Hfil(u)) = \int_u^t \Prob(T_{k+1} \in [s, s + \dif s] \,|\, \Hfil(u)) \, \dif s
  = \int_u^t f^*(s \, | \, \Hfil(u)) \, \dif s \,.
\]
The joint \ac{p.d.f.} for a realisation $\{t_1, t_2, \dots, t_k\}$ is then, by the chain rule,
\vspace{-5pt}
\begin{equation} \label{joint_dens}
  f(t_1, t_2, \dots, t_k) = \prod_{i=1}^k f^*(t_i \,|\, \Hfil(t_{i-1}))\,.
\end{equation}

In the literature the notation rarely specifies $\Hfil(\cdot)$
explicitly, but rather a superscript asterisk is used (see for
example~\cite{daley2003a}). We follow this convention and abbreviate
$F^*(t \,|\, \Hfil(u))$ and $f^*(t \,|\, \Hfil(u))$ to $F^*(t)$ and
$f^*(t)$, respectively. \new

\begin{remark}
The function $f^*(t)$ can be used to classify certain classes of point
processes. For example, if a point process has an $f^*(t)$ which is
independent of $\Hfil(t)$ then the process is a \emph{renewal process}.  \new
\end{remark}

\subsection{Conditional intensity functions}

Often it is difficult to work with the conditional arrival
distribution $f^*(t)$. Instead, another characterisation of point
processes is used: the conditional intensity function. Indeed if the
conditional intensity function exists it uniquely characterises the
finite-dimensional distributions of the point process (see Proposition
7.2.IV of \cite{daley2003a}). Originally this function was called the
hazard function \cite{cox1955} and was defined as
\begin{equation} \label{cond_int_joint_dens}
  \cInt{t} = \frac{f^*(t)}{1-F^*(t)}\,.
\end{equation}
Although this definition is valid, we prefer an intuitive
representation of the conditional intensity function as the expected
rate of arrivals conditioned on $\Hfil(t)$: \new

\begin{definition}[Conditional intensity function] \label{cif}
  Consider a counting process $N\fn$ with associated histories $\Hfil\fn$. If a (non-negative) function $\cInt{t}$ exists such that
  \[ \cInt{t} = \lim_{h \downarrow 0} \frac{\sE[N(t+h)-N(t)\,|\, \Hfil(t)]}{h} \]
  which only relies on information of $N\fn$ in the past (that is, $\cInt{t}$ is $\Hfil(t)$-measurable), then it is called the \emph{conditional intensity function} of $N\fn$.
\end{definition}

The terms `self-exciting' and `self-regulating' can be made precise by
using the conditional intensity function. If an arrival causes the
conditional intensity function to increase then the process is said to
be \emph{self-exciting}. This behaviour causes temporal clustering of
$\T$. In this setting $\cInt{t}$ must be chosen to avoid
\emph{explosion}, where we use the standard definition of explosion as
the event that $N(t)-N(s)=\infty$ for $t-s < \infty$. See
\fig{typical_cond_int} for an example realisation of such a $\cInt{t}$.
\new

\begin{figure}
  \centering
  \includegraphics[width=0.8\textwidth]{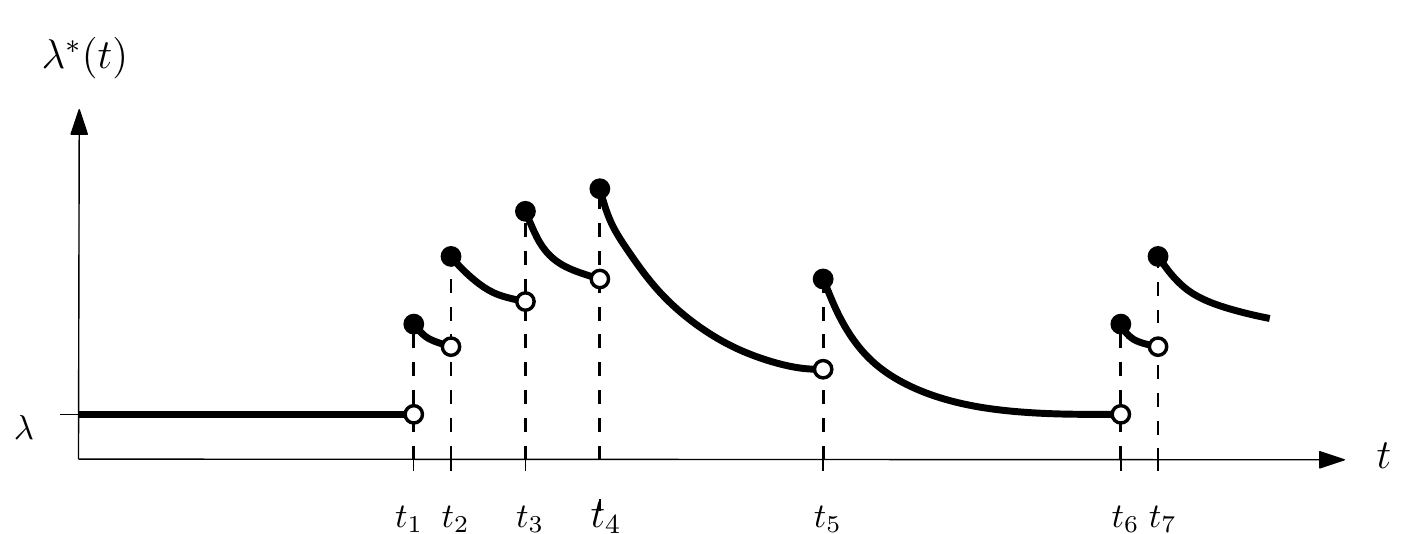}
  \caption[An example conditional intensity function for a self-exciting process]{An example conditional intensity function for a self-exciting process.}
  \label{typical_cond_int}
\end{figure}

Alternatively, if the conditional intensity function drops after an
arrival the process is called \emph{self-regulating} and the arrival
times appear quite temporally regular. Such processes are not examined
hereafter, though an illustrative example would be the arrival of
speeding tickets to a driver over time (assuming each arrival causes a
period of heightened caution when driving). \new

\subsection{Compensators}

Frequently the integrated conditional intensity function is needed
(for example, in parameter estimation and goodness of fit testing); it is
defined as follows. \new

\begin{definition}[Compensator] \label{compensator}
    For a counting process $N\fn$ the non-decreasing function
    \[ \Lambda(t) = \int_0^t \cInt{s} \,\dif s \]
    is called the \emph{compensator} of the counting process.
\end{definition}

In fact, a compensator is usually defined more generally and exists
even when $\cInt{\cdot}$ does not exist. Technically $\Lambda(t)$ is the
unique $\Hfil(t)$ predictable function, with $\Lambda(0)=0$, and is
non-decreasing, such that $N(t) = M(t) + \Lambda(t) $ almost surely
for $t \geq 0$ and where $M(t)$ is an $\Hfil(t)$ local martingale, whose
existence is guaranteed by the Doob--Meyer decomposition theorem.
However, for \acp{HP} $\cInt{\cdot}$ always exists (in fact, as we shall
see in \secn{lit_review}, a \ac{HP} is defined in terms of this
function) and therefore \dfn{compensator} is sufficient for our
purposes.

\section{Literature review} \label{lit_review}

With essential background and core concepts outlined in
Section~\ref{background}, we now turn to discussing \acp{HP},
including their useful immigration--birth representation. We briefly
touch on generalisations, before turning to a illustrative account of
\acp{HP} for financial applications. \new

\subsection{The Hawkes process}

Point processes gained a significant amount of attention in the field of statistics during the 1950s and 1960s. First, Cox \cite{cox1955} introduced the notion of a doubly stochastic Poisson process (now called the Cox process) and Bartlett \cite{bartlett1963spectral,bartlett1963density,bartlett1964} investigated statistical methods for point processes based on their power spectral densities. At IBM Research Laboratories, Lewis \cite{lewis1964} formulated a point process model (for computer failure patterns) which was a step in the direction of the \ac{HP}. The activity culminated in the significant monograph by Cox and Lewis \cite{cox1966} on time series analysis; modern researchers appreciate this text as an important development of point process theory since it canvassed their wide range of applications \cite[p.\ 16]{daley2003a}. \new

It was in this context that Hawkes \cite{hawkes1971spectra} set out to
bring Bartlett's spectral analysis approach to a new type of process:
a self-exciting point process. The process Hawkes described was a
one-dimensional point process (though originally specified for
$t\in\R$ as opposed to $t \in [0,\infty)$), and is defined as follows. \new

\begin{definition}[Hawkes process] \label{hawkes}
  Consider $(N(t) : t \geq 0)$ a counting process, with associated history $(\Hfil(t) : t \geq 0)$, that satisfies
  \vspace{-5pt}
  \[ \Prob(N(t+h)-N(t) = m \,|\, \Hfil(t)) = \begin{cases}
    \cInt{t}\, h + \mathrm{o}(h)\,,     & m=1 \\
    \mathrm{o}(h)\,,               & m>1 \\
    1 - \cInt{t}\, h + \mathrm{o}(h)\,, & m=0
  \end{cases}\,.
  \]
  Suppose the process' conditional intensity function is of the form
  \begin{equation} \label{cond_int}
    \cInt{t} = \bInt + \int_{0}^{t} \exite(t-u) \,\dif N(u)
  \end{equation}
  for some $\bInt > 0$ and $\exite:(0,\infty)\rightarrow[0,\infty)$
	  which are called the \emph{background intensity} and
	  \emph{excitation function} respectively. Assume that
	  $\exite(\cdot) \not= 0$ to avoid the trivial case, that is,
	  a homogeneous Poisson process. Such a process $N\fn$ is a \emph{\acl{HP}}.
\end{definition}

\begin{remark} \label{second_hp_def}
The definition above has $t$ as non-negative, however an alternative form of the \ac{HP} is to consider arrivals for $t \in \R$ and set $N(t)$ as the number of arrivals in $(0,t]$. Typically \ac{HP} results hold for both definitions, though we will specify that this second $t\in\R$ definition is to be used when it is required.
\end{remark}

\begin{figure}[h]
  \centering
  \sidesubfloat[]{\includegraphics[width=0.8\textwidth]{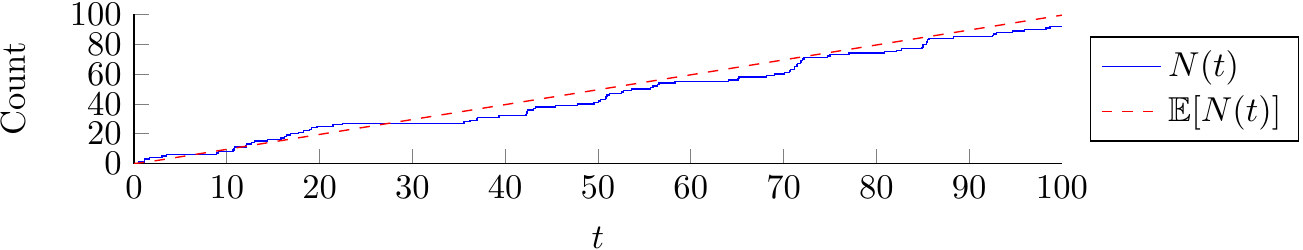}} \\
  \sidesubfloat[]{\includegraphics[width=0.8\textwidth]{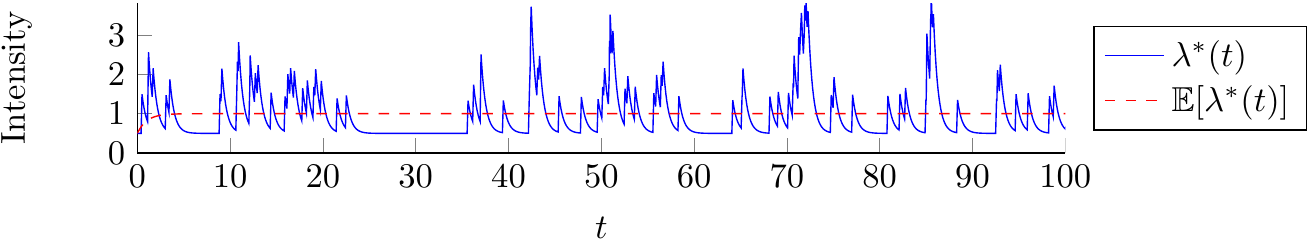}}
  \caption[A typical Hawkes process realisation $N(t)$ and associated $\cInt{t}$]{(a) A typical Hawkes process realisation $N(t)$, and its associated $\cInt{t}$ in (b), both plotted against their expected values.}
  \label{example_single_hawkes}
\end{figure}

\begin{remark}
In modern terminology, \dfn{hawkes} describes a \emph{linear} \ac{HP}---the \emph{nonlinear} version is given later in \dfn{nonlinear_hawkes}. Unless otherwise qualified, the \acp{HP} in this paper will refer to this linear form.
\end{remark}

A realisation of a \ac{HP} is shown in \fig{example_single_hawkes} with the associated path of the conditional intensity process. Hawkes \cite{hawkes1971point} soon extended this single point process into a collection of self- and mutually-exciting point processes, which we will turn to discussing after elaborating upon this one-dimensional process. \new

\subsection{Hawkes conditional intensity function}

The form of the Hawkes conditional intensity function in \eq{cond_int} is consistent with the literature though it somewhat obscures the intuition behind it. Using $\{t_1, t_2, \dots, t_k \}$ to denote the observed sequence of past arrival times of the point process up to time $t$, the Hawkes conditional intensity is
\[ \cInt{t} = \bInt + \sum_{t_i < t}~\exite(t - t_i) \,. \]
The structure of this $\cInt{\cdot}$ is quite flexible and only requires specification of the background intensity $\bInt > 0$ and the excitation function $\exite\fn$. A common choice for the excitation function is one of exponential decay; Hawkes \cite{hawkes1971spectra} originally used this form as it simplified his theoretical derivations \cite{hautsch2011}. In this case $\mu(t) = \alpha \, \e^{-\beta t}$, which is parameterised by constants $\alpha, \beta > 0$, and hence
\begin{equation} \label{exp_cond_int}
  \cInt{t} = \bInt + \int_{-\infty}^t \hspace{-5pt} \alpha \e^{-\beta(t-s)} \,\dif N(s)
        = \bInt + \sum_{t_i < t}~\alpha \e^{-\beta(t-t_i)}\,.
\end{equation}
The constants $\alpha$ and $\beta$ have the following interpretation: each arrival in the system instantaneously increases the arrival intensity by $\alpha$, then over time this arrival's influence decays at rate $\beta$. \new

Another frequent choice for $\exite\fn$ is a power law function, giving
\[
\cInt{t} = \bInt + \int_{-\infty}^{t} \frac{k}{(c+(t-s))^p} \,\dif N(s)
  = \bInt + \sum_{t_i < t}~ \frac{k}{(c+(t-t_i))^p}
\]
with some positive scalars $c, k,$ and $p$. The power law form was popularised by the geological model called Omori's law, used to predict the rate of aftershocks caused by an earthquake \cite{ogata1999}. More computationally efficient than either of these excitation functions is a piecewise linear function as in \cite{chatalbashev2007}. However, the remaining discussion will focus on the exponential form of the excitation function, sometimes referred to as the \ac{HP} with \emph{exponentially decaying intensity}. \new

One can consider the impact of setting an initial condition $\cInt{0} = \lambda_0$, perhaps in order to model a process from some time after it is started. In this scenario the conditional intensity process (using the exponential form of $\exite\fn$) satisfies the stochastic differential equation
\[ \dif\cInt{t} = \beta (\bInt - \cInt{t})\dif t + \alpha \dif N(t)\,, \quad t\geq0 \,. \]
Applying stochastic calculus yields the general solution of
\[ \cInt{t} = \e^{-\beta t} (\lambda_0 - \bInt) + \bInt + \int_0^t \alpha \e^{\beta(t-s)} \,\dif N(s)\,, \quad t\geq0 \,, \]
which is a natural extension of \eq{exp_cond_int} \cite{dafonseca2014}. \new

\subsection{Immigration--birth representation} \label{im_birth}

Stability properties of the \ac{HP} are often simpler to divine if it
is viewed as a branching process. Imagine counting the population in a
country where people arrive either via \emph{immigration} or by
\emph{birth}. Say that the stream of immigrants to the country form a
homogeneous Poisson process at rate $\bInt$. Each individual then produces
zero or more children independently of one another, and the arrival of
births form an inhomogeneous Poisson process. \new

An illustration of this interpretation can be seen in
\fig{immig_birth}. In branching theory terminology, this
\emph{immigration--birth representation} describes a Galton--Watson
process with a modified time dimension. Hawkes \cite{hawkes1974} used
the representation to derive asymptotic characteristics of the
process, such as the following result.

\begin{theorem}[Hawkes process asymptotic normality] \label{asymptotic}
  If
  \[ 0 < n := \int_0^\infty \exite(s) \, \dif s < 1 \text{ and }
  \int_0^\infty s \exite(s) \, \dif s < \infty \]
  then the number of \ac{HP} arrivals in $(0, t]$
  is asymptotically ($t \rightarrow \infty$) normally distributed.
  More precisely, writing $N(0,t] = N(t)-N(0)$,
  \[ \Prob\left( \frac{N(0,t] - \bInt t/(1-n)}{\sqrt{\bInt t/(1-n)^3}} \leq y \right)  \rightarrow \Phi(y) \,, \]
  where $\Phi\fn$ is the \ac{c.d.f.} of the standard normal distribution.
\end{theorem}

\begin{remark}
 More modern work uses the immigration--birth representation for
 applying Bayesian techniques; see, for example, \cite{rasmussen2013}. \new
\end{remark}

\begin{figure}
  \centering
  \includegraphics[width=0.8\textwidth]{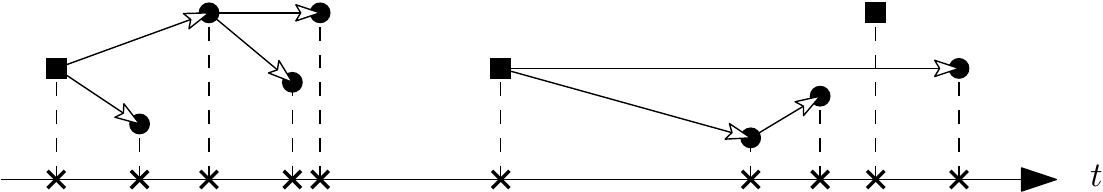}
  \caption[Example Hawkes process immigration--birth representation]{Hawkes process represented as a collection of family trees (immigration--birth representation). Squares ($\blacksquare$) indicate immigrants, circles (\CIRCLE ) are offspring/descendants, and the crosses ($\bm{\times}$) denote the generated point process.}
  \label{immig_birth}
\end{figure}

For an individual who enters the system at time $t_i \in \R$, the rate
at which they produce offspring at future times $t > t_i$ is
$\exite(t-t_i)$. Say that the direct offspring of this individual
comprise
the \emph{first-generation}, and their offspring comprise the
\emph{second-generation}, and so on; members of the union of all these
generations are called the \emph{descendants} of this $t_i$ arrival.
\new

Using the notation from \cite[Section 5.4]{grimmett2001}, define $Z_i$
to be the random number of offspring in the $i$th generation (with
$Z_0 = 1$). As the first-generation offspring arrived from a Poisson
process $Z_1 \sim \mathrm{Poi}(n)$ where the mean $n$ is known as the
\emph{branching ratio}. This branching ratio (which can take values in
$(0, \infty)$) is defined in \thrm{asymptotic} and in the case of an
exponentially decaying intensity is
\begin{equation} \label{exp_n}
    n = \int_0^\infty \alpha \e^{-\beta s}\, \dif s = \frac{\alpha}{\beta}\,.
\end{equation}

Knowledge of the branching ratio can inform development of simulation algorithms. For each immigrant $i$, the times of the first-generation offspring arrivals---conditioned on knowing the total number of them $Z_1$---are each \ac{i.i.d.} with density $\exite(t-t_i)/n$. \secn{simulation} explores \ac{HP} simulation methods inspired by the immigration--birth representation in more detail. \new

The value of $n$ also determines whether or not the \ac{HP} explodes. To see this, let $g(t) = \sE[\cInt{t}]$. A renewal-type equation will be constructed for $g$ and then its limiting value will be determined. Conditioning on the time of the first jump,
\[
    g(t) = \sE \left[ \cInt{t} \right]
         = \sE \left[ \bInt + \int_0^t \exite(t-s) \dif N(s) \right]
         = \bInt + \int_0^t \exite(t-s) \, \sE[ \dif N(s) ] \,.
\]
In order to calculate this expected value, start with
\[
  \cInt{s} = \lim_{h \downarrow 0} \frac{\sE[N(s+h)-N(s)\,|\, \Hfil(s)]}{h}
  = \frac{\sE[\dif N(s)\,|\, \Hfil(s)]}{\dif s}
\]
and take expectations (and apply the tower property)
\[
  g(s) = \sE[\cInt{s}]
  = \frac{\sE[\sE[\dif N(s)\,|\, \Hfil(s)]]}{\dif s}
  = \frac{\sE[\dif N(s)]}{\dif s}
\]
to see that
\[ \sE[\dif N(s)] = g(s) \dif s \,. \]
Therefore
\[
  g(t) = \bInt + \int_0^t \exite(t-s) \, g(s) \dif s
  = \bInt + \int_0^t g(t-s) \, \exite(s) \dif s \,.
\]
This renewal--type equation (in convolution notation is $g = \bInt + g \star \exite$) then has different solutions according to the value of $n$. Asmussen \cite{asmussen2003} splits the cases into: the \emph{defective} case ($n < 1$), the \emph{proper} case ($n=1$), and the \emph{excessive} case ($n > 1$). Asmussen's Proposition 7.4 states that for the defective case
\begin{equation} \label{mean_intensity}
    g(t) = \sE[\cInt{t}] \rightarrow \frac{\lambda}{1-n}\,,\quad\text{as }t\rightarrow \infty \,.
\end{equation}
However in the excessive case, $\cInt{t} \rightarrow \infty$
exponentially quickly, and hence $N\fn$ eventually explodes \ac{a.s.} \new

Explosion for $n>1$ is supported by viewing the arrivals as a
branching process. Since $\sE[Z_i] = n^i$ (see Section 5.4 Lemma 2 of
\cite{grimmett2001}), the expected number of descendants for one
individual is
\[ \sE\left[\sum_{i=1}^\infty Z_i \right] = \sum_{i=1}^\infty \sE[Z_i] = \sum_{i=1}^\infty n^i = \begin{cases}
  \frac{n}{1-n}, & n < 1 \\
  \infty, & n \geq 1
\end{cases}\,. \]
Therefore $n \geq 1$ means that one immigrant would generate
infinitely many descendants on average. \new

When $n \in (0,1)$ the branching ratio can be interpreted as a
probability. It is the ratio of the number of descendants for one
immigrant, to the size of their entire family (all descendants
plus the original immigrant); that is
\[ \frac{\sE\left[\sum_{i=1}^\infty Z_i \right]}{1+\sE\left[\sum_{i=1}^\infty Z_i \right]} = \frac{\frac{n}{1-n}}{1+\frac{n}{1-n}} = \frac{\frac{n}{1-n}}{\frac{1}{1-n}} = n \,. \]
Therefore, any \ac{HP} arrival selected at random was generated
\emph{endogenously} (a child) \ac{w.p.}~$n$ or
\emph{exogenously} (an immigrant) \ac{w.p.} $1-n$. Most
properties of the \ac{HP} rely on the process being \emph{stationary},
which is another way to insist that $n \in (0, 1)$ (a rigorous
definition is given in \secn{spectral_analysis}), so this is assumed
hereinafter. \new

\subsection{Covariance and power spectral densities} \label{spectral_analysis}

\acp{HP} originated from the spectral analysis of general stationary
point processes. The \ac{HP} is stationary for finite values of $t$ when it is defined as per Remark \ref{second_hp_def}, so we will use this definition for the remainder of \ssec{spectral_analysis}. Finding the power spectral density of the \ac{HP}
gives access to many techniques from the spectral analysis field; for
example, model fitting can be achieved by using the observed
periodogram of a realisation. The power spectral density is defined in
terms of the covariance density. Once again the exposition is
simplified by using the shorthand that
\[ \dif N(t) = \lim_{h \downarrow 0} N(t+h)-N(t) \,.\]

Unfortunately the term `stationary' has many different meanings in
probability theory. In this context the \ac{HP} is stationary when the
jump process $(\dif N(t) : t \geq 0)$---which takes values in $\{0,
1\}$---is \emph{weakly stationary}. This means that $\sE[\dif N(t)]$
and $\Cov(\dif N(t), \dif N(t+s))$ do not depend on $t$. Stationarity
in this sense does not imply stationarity of $N\fn$ or stationarity of
the inter-arrival times \cite{lewis1970}. One consequence of
stationarity is that $\cInt{\cdot}$ will have a long term mean (as given
by \eq{mean_intensity})
\begin{equation} \label{lm}
  \lm := \sE[\cInt{t}] = \frac{\sE[\dif N(t)]}{\dif t} = \frac{\lambda}{1-n} \,.
\end{equation}

The \emph{(auto)covariance density} is defined, for $\tau > 0$, to be
\[ R(\tau) = \Cov\left(\frac{\dif N(t)}{\dif t}, \frac{\dif N(t+\tau)}{\dif \tau}\right) \,. \]
Due to the symmetry of covariance, $R(-\tau) = R(\tau)$,
however $R\fn$ cannot be extended to the whole of $\R$ because there
is an atom at~$0$. For simple point processes $\sE[(\dif N(t))^2] = \sE[\dif
N(t)]$ (since $\dif N(t) \in \{0, 1\}$) therefore for $\tau=0$
\[ \sE[(\dif N(t))^2] = \sE[\dif N(t)] = \lm \dif t \,. \]
The \emph{complete covariance density} (complete in that its domain is all of $\R$) is defined as
\begin{equation} \label{complete}
    R^{(c)}(\tau) = \lm \delta(\tau) + R(\tau)
\end{equation}
where $\delta\fn$ is the Dirac delta function.
\begin{remark}
Typically $R(0)$ is defined such that $R^{(c)}\fn$ is everywhere continuous. Lewis \cite[p.\ 357]{lewis1970} states that strictly speaking $R^{(c)}\fn$ ``does not have a `value' at $\tau=0$''. See \cite{bartlett1963density,cox1966}, and \cite{hawkes1971spectra} for further details.
\end{remark}
The corresponding \emph{power spectral density function} is then
\begin{equation} \label{spectral_def}
  S(\omega) := \frac{1}{2\pi} \int_{-\infty}^{\infty} \e^{-i\tau \omega} R^{(c)}(\tau) \dif \tau
  = \frac{1}{2\pi} \left[ \lm + \int_{-\infty}^{\infty} \e^{-i\tau \omega} R(\tau) \dif \tau \right] \,.
\end{equation}
Up to now the discussion (excluding the final value of \eq{lm}) has
considered general stationary point processes. To apply the theory
specifically to \acp{HP} we need the following result. \new

\begin{theorem}[Hawkes process power spectral density] \label{spectral}
  Consider a \ac{HP} with an exponentially decaying intensity with $\alpha < \beta$. The intensity process then has covariance density, for $\tau > 0$,
  \[ R(\tau) = \frac{\alpha \beta \bInt (2\beta - \alpha)}{2(\beta-\alpha)^2} \e^{-(\beta-\alpha)\tau} \,. \]
  Hence, its power spectral density is, $\forall \, \omega \in \R$,
  \[ S(\omega) = \frac{\bInt \beta}{2\pi(\beta-\alpha)} \left( 1 + \frac{\alpha(2\beta - \alpha)}{(\beta-\alpha)^2 + \omega^2} \right) \,. \]
\end{theorem}

\begin{proof} (Adapted from \cite{hawkes1971spectra}.) Consider the covariance density for $\tau \in \R \setminus \{0\}$:
\begin{equation} \label{cov}
    R(\tau) = \sE\left[\frac{\dif N(t)}{\dif t}\frac{\dif N(t+\tau)}{\dif \tau}\right] - \lm^2 \,.
\end{equation}
Firstly note that, via the tower property,
\begin{align*}
  \sE\left[\frac{\dif N(t)}{\dif t}\frac{\dif N(t+\tau)}{\dif \tau} \right]
  &= \sE \left[ \sE\left[ \frac{\dif N(t)}{\dif t}\frac{\dif N(t+\tau)}{\dif \tau} \,\Big|\,\Hfil(t+\tau) \right]  \right] \\
  &= \sE \left[ \frac{\dif N(t)}{\dif t} \sE\left[ \frac{\dif N(t+\tau)}{\dif \tau} \,\Big|\,\Hfil(t+\tau) \right]  \right] \\
  &= \sE \left[ \frac{\dif N(t)}{\dif t} \cInt{t+\tau} \right]\,.
\end{align*}
Hence \eq{cov} can be combined with \eq{cond_int} to see that $R(\tau)$ equals
\[
  \sE\left[ \frac{\dif N(t)}{\dif t} \left( \bInt + \int_{-\infty}^{t+\tau} \mu(t+\tau-s) \dif N(s) \right)\right] - \lm^2 ,
\]
which yields
\begin{align} \label{to_transform}
  R(\tau) &= \lm \exite(\tau) + \int_{-\infty}^{\tau} \exite(\tau-v) R(v) \dif v \notag \\
  &= \lm \exite(\tau) + \int_0^{\infty} \exite(\tau + v) R(v) \dif v + \int_0^{\tau} \exite(\tau - v) R(v) \dif v \,.
\end{align}
Refer to Appendix \ref{part_i} for details; this is a Wiener--Hopf-type
integral equation. Taking the Laplace transform of \eq{to_transform} gives
\[
  \Laplace{R(\tau)}(s) = \frac{\alpha \lm ( 2\beta - \alpha)}{2(\beta - \alpha)(s + \beta - \alpha)} \,.
\]
Refer to Appendix \ref{part_ii} for details.
Note that \eq{exp_n} and \eq{lm} supply
$\lm = {\beta\lambda}/{(\beta-\alpha)}$, which implies that
\[
\Laplace{R(\tau)}(s) = \frac{\alpha \beta\lambda( 2\beta - \alpha)}{2(\beta
- \alpha)^2(s + \beta - \alpha)}\,.
\]
Therefore,
\[
R(\tau) = \Linv{\frac{\alpha \beta\lambda( 2\beta - \alpha)}{2(\beta - \alpha)^2(s + \beta - \alpha)}}
    = \frac{\alpha \beta \bInt (2\beta - \alpha)}{2(\beta-\alpha)^2} \e^{-(\beta-\alpha)\tau} \,.
\]
The values of $\lm$ and $\Laplace{R(\tau)}(s)$ are then substituted into the definition given in \eq{spectral_def}:
\begin{align*}
  S(\omega) &= \frac{1}{2\pi} \left[ \lm + \int_{-\infty}^{\infty} \e^{-i\tau \omega} R(\tau) \dif \tau \right] \\
    &= \frac{1}{2\pi} \left[ \lm + \int_{0}^{\infty} \e^{-i\tau \omega} R(\tau) \dif \tau + \int_{0}^{\infty} \e^{i\tau \omega} R(\tau) \dif \tau \right] \\
    &= \frac{1}{2\pi} \left[ \lm + \Laplace{R(\tau)}(i\omega) + \Laplace{R(\tau)}(-i\omega) \right] \\
    &= \frac{1}{2\pi} \left[ \lm + \frac{\alpha \lm ( 2\beta - \alpha)}{2(\beta - \alpha)(i\omega + \beta - \alpha)} + \frac{\alpha \lm ( 2\beta - \alpha)}{2(\beta - \alpha)(-i\omega + \beta - \alpha)} \right] \\
    &= \frac{\bInt \beta}{2\pi(\beta - \alpha)} \left[ 1 + \frac{\alpha ( 2\beta - \alpha)}{(\beta-\alpha)^2 + \omega^2} \right] \,.
\end{align*}
\end{proof}
\begin{remark}
The power spectral density appearing in Theorem~\ref{spectral} is a shifted scaled Cauchy \ac{p.d.f.}
\end{remark}

\begin{remark}
As $R\fn$ is a real-valued symmetric function, its Fourier transform
$S\fn$ is also real-valued and symmetric, that is,
\[
  S(\omega) = \frac{1}{2\pi} \left[ \lm + \int_{-\infty}^{\infty} \e^{-i\tau \omega} R(\tau) \dif \tau \right]
  = \frac{1}{2\pi} \left[ \lm + \int_{-\infty}^{\infty} \cos(\tau \omega) R(\tau) \dif \tau \right]\,,
\]
and
\[ S_+(\omega) := S(-\omega)+S(\omega) = 2S(\omega) \,. \]
It is common that $S_+\fn$ is plotted instead of $S\fn$, as in Section 4.5 of \cite{cox1966}; this is equivalent to wrapping the negative frequencies over to the positive half-line. \new
\end{remark}

\subsection{Generalisations}

The immigration--birth representation is useful both theoretically and
practically. However it can only be used to describe \emph{linear}
\acp{HP}. Br\'emaud and Massouli\'e \cite{bremaud1996} generalised the
\ac{HP} to its nonlinear form: \new

\begin{definition}[Nonlinear Hawkes process] \label{nonlinear_hawkes}
  Consider a counting process with conditional intensity function of the form
  \[ \cInt{t} = \Psi\left( \int_{-\infty}^t \exite(t-s)\, N(\dif s) \right) \]
  where $\Psi: \R \rightarrow [0,\infty)$, $\exite: (0, \infty) \rightarrow \R$. Then $N\fn$ is a \emph{nonlinear Hawkes process}. Selecting $\Psi(x) = \bInt + x$ reduces $N\fn$ to the linear \ac{HP} of \dfn{hawkes}.
\end{definition}

Modern work on nonlinear \acp{HP} is much rarer than the original
linear case (for simulation see pp.\ 96--116 of \cite{carstensen2010},
and associated theory in \cite{zhu2013}). This is due to a combination
of factors; firstly, the generalisation was introduced relatively
recently, and secondly, the increased complexity frustrates even
simple investigations. \new

Now to return to the extension mentioned earlier, that of a collection
of self- and \emph{mutually-exciting} \acp{HP}. The processes being
examined are collections of one-dimensional \acp{HP} which `excite'
themselves and each other.

\begin{definition}[Mutually exciting Hawkes process] \label{mutually_exciting}
Consider a collection of $m$ counting processes $\{N_1(\cdot)$, $\dots,
N_m(\cdot)\}$ denoted $\Ns$. Say $\{T_{i,j} : i \in \{1, \dots, m\}, j
\in \Nat\}$ are the random arrival times for each counting process
(and $t_{i,j}$ for observed arrivals). If for each $i=1,\dots,m$ then
$N_i\fn$ has conditional intensity of the form
  \begin{equation} \label{multi_cond_int}
    \cInt[i]{t} = \bInt_i + \sum_{j=1}^m \int_{-\infty}^{t} \exite_j(t-u) \,\dif N_j(u)
  \end{equation}
  for some $\bInt_i > 0$ and $\exite_i: (0,\infty)\rightarrow[0,\infty)$, then $\Ns$ is called a \emph{mutually exciting Hawkes process}.
\end{definition}

When the excitation functions are set to be exponentially decaying,
\eq{multi_cond_int} can be written as
\begin{equation} \label{multi_exp_cond_int}
  \cInt[i]{t} = \bInt_i + \sum_{j=1}^m \int_{-\infty}^{t} \alpha_{i,j} \e^{-\beta_{i,j}(t-s)} \,\dif N_j(s)
         = \bInt_i + \sum_{j=1}^m \sum_{t_{j,k} < t} \hspace{-5pt} \alpha_{i,j} \e^{-\beta_{i,j}(t-t_{j,k})}
\end{equation}
for non-negative constants $\{ \alpha_{i,j}, \beta_{i,j} : i,j = 1, \dots, m \}$.

\begin{remark}
There are models for \acp{HP} where the points themselves are
multi-dimensional, for example, spatial \acp{HP} or temporo-spatial \acp{HP}
\cite{mohler2011}. One should not confuse mutually exciting \acp{HP} with
these multi-dimensional \acp{HP}. \new
\end{remark}

\subsection{Financial applications} \label{financial}

This section reviews primarily the work of A{\"\i}t-Sahalia, et
al.~\cite{ait2010} and Filimonov and Sornette \cite{filimonov2012}.
It assumes the reader is familiar with mathematical finance and the use
of stochastic differential equations. \new

\subsubsection{Financial contagion}

We turn our attention to the moest recent
applications of \acp{HP}. A major domain for self- and
mutually-exciting processes is financial analysis. Frequently it is
seen that large movements in a major stock market propagate in foreign
markets as a process called \emph{financial contagion}. Examples of
this phenomenon are clearly visible in historical series of asset
prices; \fig{financial_contagion} illustrates one such case. \new

The `Hawkes diffusion model' introduced by \cite{ait2010} is an
attempt to extend previous models of stock prices to include financial
contagion. Modern models for stock prices are typically built upon the
model popularised by \cite{black1973} where the log returns on the
stock follow geometric Brownian motion. Whilst this seminal paper was
lauded by the economics community, the model inadequately captured the
`fat tails' of the return distribution and so was not commonly used by
traders \cite{haug2009}. Merton \cite{merton1976} attempted to
incorporate heavy tails by including a Poisson jump process to model
booms and crashes in the stock returns; this model is often called
Merton diffusion model. The Hawkes diffusion model extends this model
by replacing the Poisson jump process with a mutually-exciting
\ac{HP}, so that crashes can self-excite and propagate in a market and
between global markets. \new

The basic Hawkes diffusion model describes the log returns of $m$
assets $\{X_1(\cdot), \dots, X_m(\cdot)\}$ where each asset $i =
1,\dots,m$ has associated expected return $\mu_i \in \R$, constant
volatility $\sigma_i \in \R^+$, and standard Brownian motion
$(W_i^X(t) : t \geq 0)$. The Brownian motions have constant
correlation coefficients $\{\rho_{i,j} : i,j = 1, \dots, m\}$. Jumps
are added by a self- and mutually-exciting \ac{HP} (as per
\dfn{mutually_exciting} with some selection of constants
$\alpha_{\cdot,\cdot}$ and $\beta_{\cdot,\cdot}$) with stochastic jump
sizes $(Z_i(t) : t \geq 0)$. The asset dynamics are then assumed to
satisfy
\[ \dif X_i(t) = \mu_i \dif t + \sigma_i \dif W_i^X(t) + Z_i(t) \dif N_i(t)\,. \]

The general Hawkes diffusion model replaces the constant volatilities
with stochastic volatilities $\{V_1(\cdot),$ $\dots, V_m(\cdot)\}$
specified by the Heston model. Each asset $i = 1,\dots,m$ has a:
long-term mean volatility $\theta_i > 0$, rate of returning to this
mean $\kappa_i > 0$, volatility of the volatility $\nu_i > 0$, and
standard Brownian motion $(W_i^V(t) : t \geq 0)$. Correlation between
the $W_{\cdot}^X(\cdot)$'s is optional, yet the effect would be
dominated by the jump component. Then the full dynamics are captured by
\[ \dif X_i(t) = \mu_i \dif t + \sqrt{V_i(t)} \dif W_i^X(t) + Z_i(t) \dif N_i(t)\,, \]
\[ \dif V_i(t) = \kappa_i(\theta_i - V_i(t)) \dif t + \nu_i \sqrt{V_i(t)} \dif W_i^V(t) \,. \]

However the added realism of the Hawkes diffusion model comes at a
high price. The constant volatility model requires $5m + 3m^2$
parameters to be fit (assuming $Z_i(\cdot)$ is characterised by two
parameters) and the stochastic volatility extension requires an extra
$3m$ parameters (assuming $\forall i,j=1,\dots,m$ that
$\sE[W_i(\cdot)^VW_j(\cdot)^V]=0$). In \cite{ait2010} hypothesis tests
reject the Merton diffusion model in favour of the Hawkes diffusion
model, however there are no tests for overfitting the data (for
example, Akaike or Bayesian information criterion comparisons). Remember that
John Von Neumann (reputedly) claimed that ``with four parameters I can
fit an elephant'' \cite{dyson2004}. \new

For computational necessity the authors made a number of
simplifying assumptions to reduce the number of parameters to fit
(such as the background intensity of crashes is the same for all
markets). Even so, the Hawkes diffusion model was only able to be
fitted for pairs of markets ($m=2$) instead of for the globe as a
whole. Since the model was calibrated to daily returns of market
indices, historical data was easily available (for exmaple, from Google or
Yahoo!\ finance); care had to be taken to convert timezones and handle
the different market opening and closing times. The parameter
estimation method used by \cite{ait2010} was the generalised method of
moments, however the theoretical moments derived satisfy long and
convoluted equations. \new

\begin{figure}
  \centering
  \includegraphics[width=0.8\textwidth]{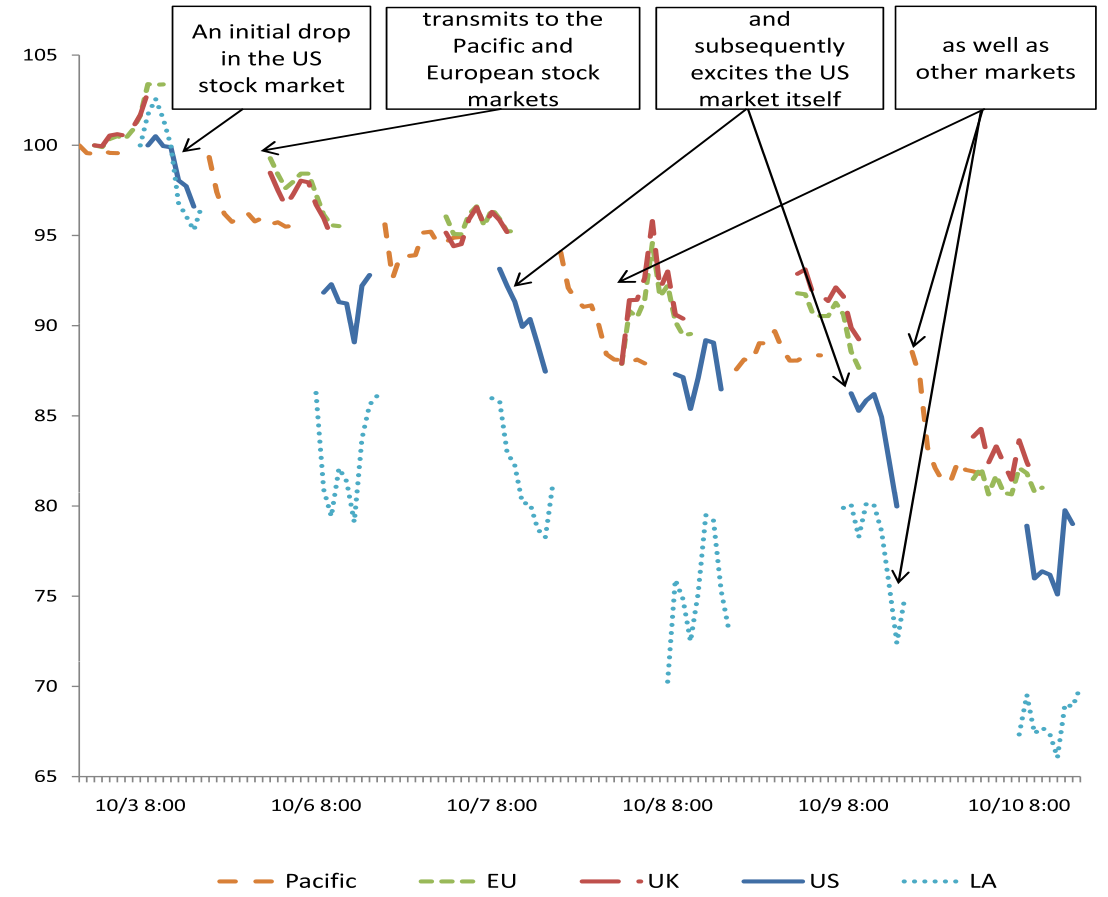}
  \caption[Example of mutual excitation in global markets]{Example of mutual excitation in global markets. This figure plots the cascade of declines in international equity markets experienced between October 3, 2008 and October 10, 2008 in the US; Latin America (LA); UK; Developed European countries (EU); and Developed countries in the Pacific. Data are hourly. The first observation of each price index series is normalised to 100 and the following observations are normalised by the same factor. Source: MSCI MXRT international equity indices on Bloomberg (reproduced from \protect\cite{ait2010}).}
  \label{financial_contagion}
\end{figure}

\subsubsection{Mid-price changes and high-frequency trading}

A simpler system to model is a single stock's price over time, though
there are many different prices to consider. For each stock one could
use: the last transaction price, the best ask price, the best bid
price, or the mid-price (defined as the average of best ask and best
bid prices). The last transaction price includes inherent
microstructure noise (for example, the bid--ask bounce), and the best ask and
bid prices fail to represent the actions of both buyers and sellers in
the market. \new

Filimonov and Sornette \cite{filimonov2012} model the mid-price
changes over time as a \ac{HP}. In particular they look at long-term
trends of the (estimated) branching ratio. In this context, $n$
represents the proportion of price moves that are not due to external
market information but simply reactions to other market participants.
This ratio can be seen as the quantification of the principle of
economic reflexivity. The authors conclude that the branching ratio
has increased dramatically from $30\%$ in 1998 to $70\%$ in 2007. \new

Later that year \cite{lorenzen2012} critiqued the test procedure used
in this analysis. Filimonov and Sornette \cite{filimonov2012} had
worked with a dataset with timestamps accurate to a second, and this
often led to multiple arrivals nominally at the same time (which is an
impossible event for simple point processes). Fake precision was
achieved by adding Unif$(0,1)$ random fractions of seconds to all
timestamps, a technique also used by \cite{bowsher2007}. Lorenzen
found that this method added an element of smoothing to the data which
gave it a better fit to the model than the actual millisecond
precision data. The randomisation also introduced bias to the \ac{HP}
parameter estimates, particularly of $\alpha$ and $\beta$. Lorenzen
formed a crude measure of high-frequency trading activity leading to
an interesting correlation between this activity and $n$ over the
observed period. \new

\begin{remark}
  Fortunately we have received comments from referees suggesting other very important works to consider, which we will briefly list here. The importance of Bowsher \cite{bowsher2007} is highlighted, as is the series by Chavez-Demoulin et al.\ \cite{chavez2005,chavez2012}. They point to the book by McNeil et al.\ \cite{mcneil2015} where a section is devoted to \ac{HP} applications, and stress the relevance of the Parisian school in applying \acp{HP} to microstructure modelling, for example, the paper by Bacry et al.\ \cite{bacry2013}.
\end{remark}

\section{Parameter estimation} \label{parameter_estimation}

This section investigates the problem of generating parameters estimates $\hat{\thetas} = (\hat{\bInt}, \hat{\alpha}, \hat{\beta})$ given some finite set of arrival times $\bft= \{t_1, t_2, \dots, t_k\}$ presumed to be from a \ac{HP}.
For brevity, the notation here will omit the $\hat{\thetas}$ and $\bft$ arguments from functions: \ $L = L(\hat{\thetas};\,\bft)$, $l = l(\hat{\thetas};\,\bft)$, $\cInt{t}=\cInt{t;\,\bft, \hat{\thetas}}$, and $\Lambda(t) = \Lambda(t;\,\bft, \hat{\thetas})$.
The estimators are tested over simulated data, for the sake of simplicity and lack of relevant data. Unfortunately this method bypasses the many significant challenges raised by real datasets, challenges that caused \cite{filimonov2013} to state that
\begin{quote}
\emph{``Our overall conclusion is that calibrating the Hawkes process is akin to an excursion
within a minefield that requires expert and careful testing before any conclusive step
can be taken.''}
\end{quote}

The method considered is \emph{maximum likelihood estimation}, which begins by finding the \emph{likelihood function}, and estimates the model parameters as the inputs which maximise this function.

\subsection{Likelihood function derivation}

Daley and Vere-Jones \cite[Proposition 7.2.III]{daley2003a} give the
following result. \new

\begin{theorem}[Hawkes process likelihood]
	Let $N\fn$ be a regular point process on $[0, T]$ for some finite positive $T$, and let $t_1, \dots, t_k$ denote a realisation of $N\fn$ over $[0, T]$. Then, the likelihood $L$ of $N\fn$ is expressible in the form
	\[ L = \Big[\prod_{i=1}^k \cInt{t_i} \Big]\exp\Big(-\int_0^T \cInt{u} \dif u \Big)\,. \]
\end{theorem}

\begin{proof}
First assume that the process is observed up to the time of
the $k$th arrival.
The joint density function from \eq{joint_dens} is
\[ L = f(t_1, t_2, \dots, t_k) = \prod_{i=1}^k f^*(t_i)\,. \]
This function can be written in terms of the conditional intensity function. Rearrange \eq{cond_int_joint_dens} to find $f^*(t)$ in terms of $\cInt{t}$ (as per \cite{rasmussen2009}):
\[
	\cInt{t} = \frac{f^*(t)}{1-F^*(t)}
		= \frac{\frac{\dif}{\dif t}F^*(t)}{1-F^*(t)}
		= -\frac{\dif \, \log(1 - F^*(t))}{\dif t} \,.
\]
Integrate both sides over the interval $(t_k, t)$:
\[ - \int_{t_k}^t \cInt{u} \dif u = \log(1 - F^*(t)) - \log(1-F^*(t_k))\,. \]
The \ac{HP} is a \emph{simple} point process, meaning that multiple
arrivals cannot occur at the same time. Hence $F^*(t_k)=0$ as $T_{k+1}
> t_k$, and so
\begin{equation} \label{generating_eq}
	- \int_{t_k}^t \cInt{u} \dif u = \log(1 - F^*(t)) \,.
\end{equation}
Further rearranging yields
\begin{equation} \label{fstar}
F^*(t) = 1 - \exp \Big( - \int_{t_k}^t \cInt{u} \dif u \Big), \qquad
f^*(t) = \cInt{t} \exp \Big(- \int_{t_k}^t \cInt{u} \dif u \Big) \,.
\end{equation}
Thus the likelihood becomes
\begin{equation} \label{likelihood}
	L = \prod_{i=1}^k f^*(t_i)
	  = \prod_{i=1}^k \cInt{t_i} \exp\Big(-\int_{t_{i-1}}^{t_i} \cInt{u} \dif u \Big)
 	  = \Big[\prod_{i=1}^k \cInt{t_i} \Big] \exp\Big(-\int_0^{t_k} \cInt{u} \dif u \Big) \,.
\end{equation}

Now suppose that
the process is observed over some time period $[0, T] \supset [0,
t_k]$.
The likelihood will then include the probability of seeing no arrivals
in the time interval $(t_k, T]$:
\[ L = \Big[\prod_{i=1}^k f^*(t_i)\Big](1-F^*(T))\,. \]
Using the formulation of $F^*(t)$ from \eq{fstar}, then
\[ L = \Big[\prod_{i=1}^k \cInt{t_i} \Big]\exp\Big(-\int_0^T \cInt{u} \dif u \Big) \,. \]
The completes the proof.
\end{proof}

\subsection{Simplifications for exponential decay}

With the likelihood function from \eq{likelihood}, the log-likelihood for the interval $[0, t_k]$ can be derived as
\begin{equation} \label{log_likelihood}
	l = \sum_{i=1}^k \log(\cInt{t_i}) - \int_{0}^{t_k} \lambda^*(u) \dif u
	= \sum_{i=1}^k \log(\cInt{t_i}) - \Lambda(t_k) \,.
\end{equation}
Note that the integral over $[0, t_k]$ can be broken up into the segments $[0, t_1]$, $(t_1, t_2]$, $\dots$, $(t_{k-1}, t_k]$, and therefore
\[
 \Lambda(t_k) = \int_0^{t_k} \cInt{u} \dif u
  = \int_0^{t_1} \cInt{u} \dif u + \sum_{i=1}^{k-1} \int_{t_i}^{t_{i+1}} \cInt{u} \dif u \,.
\]
This can be simplified in the case where $\cInt{\cdot}$ decays exponentially:
\begin{align*}
	\Lambda(t_k) &= \int_0^{t_1} \hspace{-3pt} \bInt \dif u + \sum_{i=1}^{k-1} \int_{t_i}^{t_{i+1}} \hspace{-5pt}  \bInt + \hspace{-2pt}  \sum_{t_j < u} \hspace{-2pt}  \alpha \e^{-\beta (u-t_j)} \dif u \\
	&= \bInt t_k + \alpha \sum_{i=1}^{k-1} \int_{t_i}^{t_{i+1}} \sum_{j=1}^i \e^{-\beta (u-t_j)} \dif u \\
	&= \bInt t_k + \alpha \sum_{i=1}^{k-1} \sum_{j=1}^i \int_{t_i}^{t_{i+1}} \e^{-\beta (u-t_j)} \dif u \\
	&= \bInt t_k - \frac{\alpha}{\beta} \sum_{i=1}^{k-1} \sum_{j=1}^i \left[ \e^{-\beta (t_{i+1}-t_j)} - \e^{-\beta (t_i-t_j)} \right] \,.
\end{align*}
Finally, many of the terms of this double summation cancel out leaving
\begin{equation} \label{hawkes_compensator}
	\Lambda(t_k) = \bInt t_k - \frac{\alpha}{\beta} \sum_{i=1}^{k-1} \left[ \e^{-\beta (t_k-t_i)} - \e^{-\beta (t_i-t_i)} \right]
		= \bInt t_k - \frac{\alpha}{\beta} \sum_{i=1}^k \left[ \e^{-\beta (t_k-t_i)} - 1 \right] \,.
\end{equation}
Note that here the final summand is unnecessary, though it is often included, see \cite{lorenzen2012}.
Substituting $\cInt{\cdot}$ and $\Lambda(\cdot)$ into \eq{log_likelihood} gives
\begin{equation} \label{log_like_full}
	l = \sum_{i=1}^k \log\Big[\bInt + \alpha\sum_{j=1}^{i-1} \e^{-\beta (t_i-t_j)} \Big] - \bInt t_k + \frac{\alpha}{\beta} \sum_{i=1}^k \left[ \e^{-\beta(t_k-t_i)} - 1 \right] \,.
\end{equation}

This direct approach is computationally infeasible as the first term's
double summation entails $\Oh(k^2)$ complexity. Fortunately the similar
structure of the inner summations allows $l$ to be computed with
$\Oh(k)$ complexity \cite{ogata1978,crowley2013}.
For $i \in \{2, \dots, k\}$, let $A(i) =
\sum_{j=1}^{i-1} \e^{-\beta(t_i-t_j)}$, so that
\begin{equation} \label{recursive}
	A(i)
	  = \e^{-\beta t_i + \beta t_{i-1}} \sum_{j=1}^{i-1} \e^{-\beta t_{i-1} + \beta t_j}
	  = \e^{-\beta (t_i - t_{i-1})} \Big( 1 + \sum_{j=1}^{i-2} \e^{-\beta (t_{i-1} - t_j)} \Big)
	  = \e^{-\beta (t_i - t_{i-1})} (1 + A(i-1)) \,.
\end{equation}
With the added base case of $A(1) = 0$, $l$ can be rewritten as
\begin{equation} \label{log_like_rec}
	l = \sum_{i=1}^k \log(\bInt + \alpha A(i)) - \bInt t_k + \frac{\alpha}{\beta} \sum_{i=1}^k \left[ \e^{-\beta(t_k-t_i)} - 1 \right] \,.
\end{equation}

Ozaki \cite{ozaki1979} also gives the partial derivatives and the Hessian for this log-likelihood function. Of particular note is that each derivative calculation can be achieved in order $\Oh(k)$ complexity when a recursive approach (similar to \eq{recursive}) is taken \cite{ogata1981}.
\begin{remark}
 The recursion implies that the joint process $(N(t), \cInt{t})$ is Markovian (see Remark 1.22 of \cite{liniger2009}). \new
\end{remark}

\subsection{Discussion}

Understanding of the maximum likelihood estimation method for the
\ac{HP} has changed significantly over time. The general form of the
log-likelihood function \eq{log_likelihood} was known by
Rubin~\cite{rubin1972}. It was applied to the \ac{HP} by
Ozaki~\cite{ozaki1979} who derived \eq{log_like_full} and the improved
recursive form~\eq{log_like_rec}. Ozaki also found (as noted earlier)
an efficient method for calculating the derivatives and the Hessian
matrix. Consistency, asymptotic normality and efficiency of the
estimator were proved by Ogata~\cite{ogata1978}. \new

It is clear that
the maximum likelihood estimation will usually be very effective
for model fitting. However, \cite{filimonov2012} found that,
for small samples, the estimator produces significant bias,
encounters many local optima, and is highly sensitive to the selection
of excitation function. \new
Additionally, the $\Oh(k)$ complexity can render the method useless when
samples become large; remember that any iterative optimisation routine
would calculate the likelihood function perhaps thousands of times.
The R `hawkes' package thus implements this routine in C++ in an
attempt to mitigate the performance issues. \new

This `performance bottleneck' is largely the cause of the latest trend
of using the generalised method of moments to perform parameter
estimation. Da Fonseca and Zaatour \cite{dafonseca2014} state that the
procedure is ``instantaneous'' on their test sets. The method uses
sample moments and the sample autocorrelation function which are
smoothed via a (rather arbitrary) user-selected procedure. \new

\section{Goodness of fit} \label{goodness}

This section outlines approaches to determining the
appropriateness of a \acp{HP} model for point data, which is
a critical link in their application.

\new

\subsection{Transformation to a Poisson process} \label{transformation}

Assessing the goodness of fit for some point data to a Hawkes model is an important practical consideration. In performing this assessment the point process' compensator is essential, as is the random time change theorem (here adapted from \cite{brown2002}): \new

\begin{theorem}[Random time change theorem] \label{time_change}
	Say $\{t_1, t_2, \dots, t_k\}$ is a realisation over time $[0, T]$ from a point process with conditional intensity function $\cInt{\cdot}$. If $\cInt{\cdot}$ is positive over $[0, T]$ and $\Lambda(T) < \infty$ \ac{a.s.} then the transformed points $\{ \Lambda(t_1), \Lambda(t_2), \dots, \Lambda(t_k) \}$ form a Poisson process with unit rate.
\end{theorem}

The random time change theorem is fundamental to the model fitting
procedure called \emph{(point process) residual analysis}. Original
work \cite{embrechts2011} on residual analysis goes back to
\cite{meyer1971}, \cite{papangelou1972}, and \cite{watanabe1964}.
Daley and Vere-Jones's Proposition 7.4.IV \cite{daley2003a} rewords
and extends the theorem as follows.\new

\begin{theorem}[Residual analysis] \label{residual_analysis}
	Consider an unbounded, increasing sequence of time points  $\{t_1, t_2, \dots \}$ in the half-line $(0, \infty)$, and a monotonic, continuous compensator $\Lambda(\cdot)$ such that $\lim_{t \rightarrow \infty} \Lambda(t) = \infty$ \ac{a.s.} The transformed sequence $\{t^*_1, t^*_2, \dots\} = \{\Lambda(t_1), \Lambda(t_2), \dots\}$, whose counting process is denoted $N^*(t)$, is a realisation of a unit rate Poisson process if and only if the original sequence $\{t_1, t_2, \dots \}$ is a realisation from the point process defined by $\Lambda(\cdot)$.
\end{theorem}

Hence, equipped with a closed form of the compensator from
\eq{hawkes_compensator}, the quality of the statistical inference can
be ascertained using standard fitness tests for Poisson processes.
\fig{transformed} shows a realisation of a \ac{HP} and the
corresponding transformed process. In \fig{transformed} $\Lambda(t)$
appears identical to $N(t)$. They are actually slightly different
($\Lambda\fn$ is continuous) however the similarity is expected due to
Doob--Meyer decomposition of the compensator.

\begin{figure}
    \centering
    \sidesubfloat[]{\includegraphics[width=0.4\textwidth]{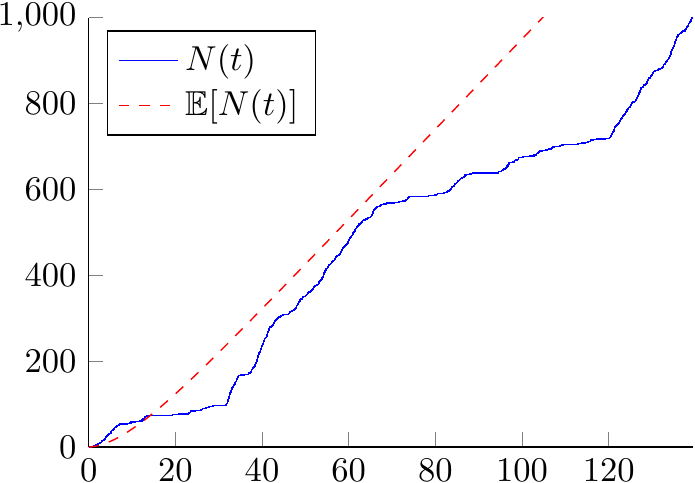}}\hspace{2em}
    \sidesubfloat[]{\hspace{11pt}\includegraphics[width=0.375\textwidth]{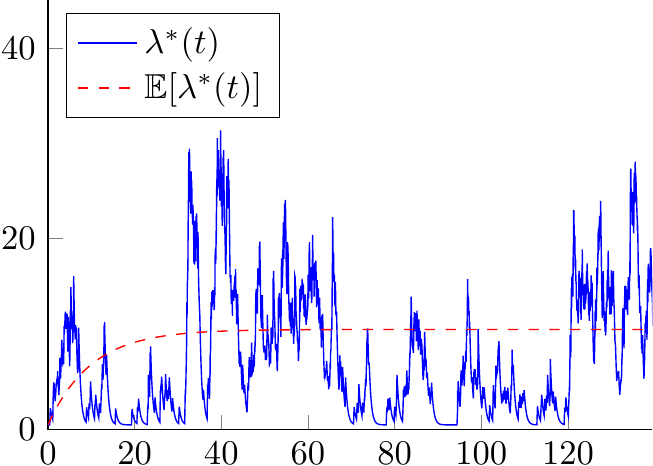}} \\
    \sidesubfloat[]{\hspace{2pt}\includegraphics[width=0.4\textwidth]{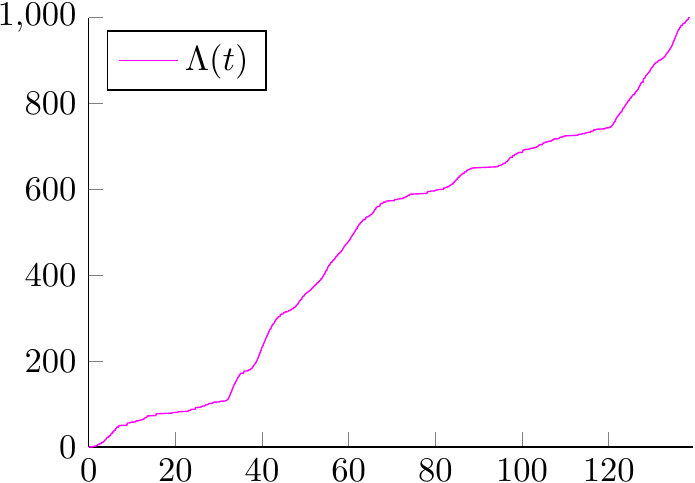}}\hspace{2em}
    \sidesubfloat[]{\includegraphics[width=0.4\textwidth]{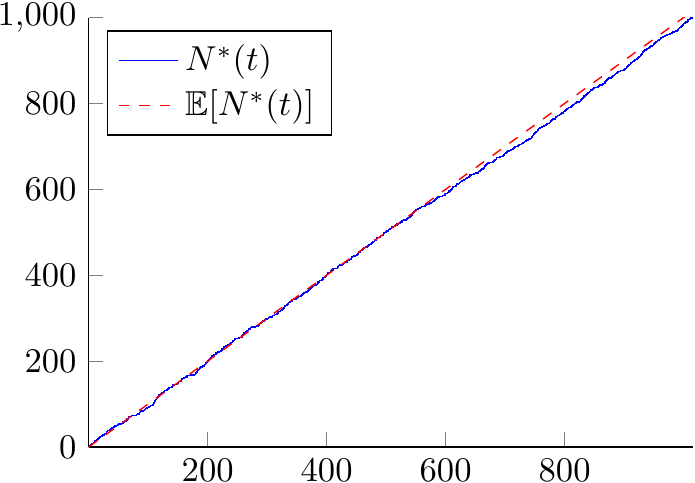}}
    \caption{An example of using the random time change theorem to transform a Hawkes process into a unit rate Poisson process. (a) A Hawkes process $N(t)$ with $(\bInt, \alpha, \beta) = (0.5, 2, 2.1)$, with the associated (b) conditional intensity function and (c) compensator. (d) The transformed process $N^*(t)$, where $t_i^* = \Lambda(t_i)$.}
    \label{transformed}
\end{figure}

\subsection{Tests for Poisson process}

\subsubsection{Basic tests}

There are many procedures for testing whether a series of points form
a Poisson process (see \cite{cox1966} for an extensive treatment). As
a first test, one can run a hypothesis test to check $\sum_i
\1_{\{t^*_i < t\}} \sim \mathrm{Poi}(t)$. If this initial test
succeeds then the interarrival times,
\[ \{\tau_1, \tau_2, \tau_3, \dots\} = \{t^*_1, t^*_2-t^*_1, t^*_3-t^*_2, \dots\}, \]
should be tested to ensure $\tau_i \iid \mathrm{Exp}(1)$. A
qualitative approach is to create a \ac{Q--Q} plot for $\tau_i$ using
the exponential distribution (see for example \fig{qq}). Otherwise a
quantitative alternative is to run Kolmogorov--Smirnov (or perhaps
Anderson--Darling) tests. \new

\subsubsection{Test for independence}

The next test, after confirming there is reason to believe that the
$\tau_i$ are exponentially distributed, is to check their
independence. This can be done by looking for autocorrelation in the
$\tau_i$ sequence. Obviously zero autocorrelation does not imply
independence, but a non-zero amount would certainly imply a
non-Poisson model. A visual examination can be conducted by plotting
the points $(U_{i+1}, U_i)$. If there are noticeable patterns then the
$\tau_i$ are autocorrelated. Otherwise the points should look evenly
scattered; see for example~\fig{autocorrelation}. Quantitative
extensions exist; for example see Section~3.3.3 of~\cite{knuth2014},
or serial correlation tests in~\cite{kroese2011}. \new

\begin{figure}[h]
  \centering
  \sidesubfloat[]{\includegraphics[width=0.4\textwidth]{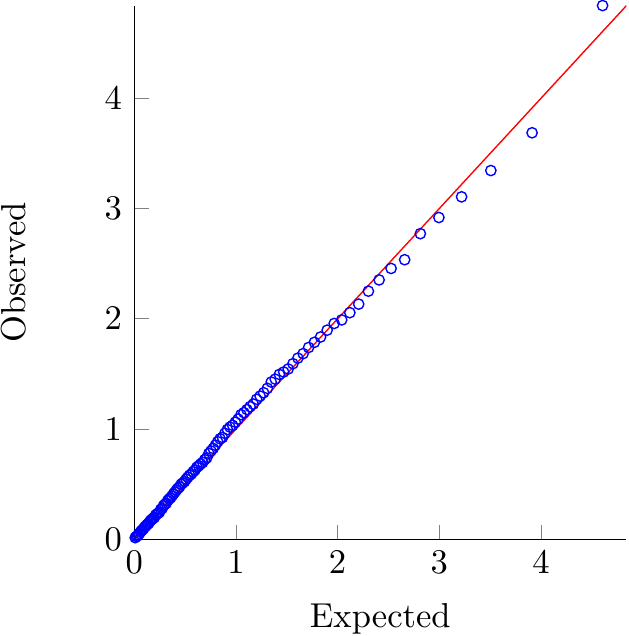}\label{qq}}\hspace{3em}
  \sidesubfloat[]{\includegraphics[width=0.4\textwidth]{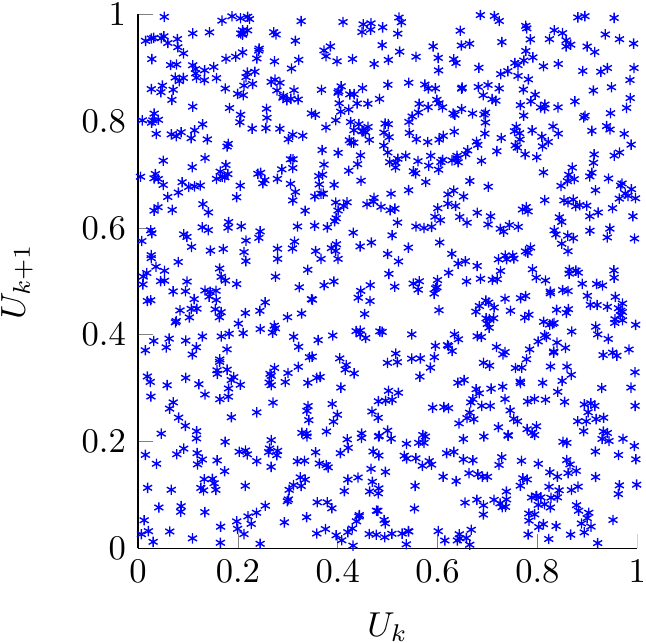}\label{autocorrelation}}
  \caption{(a) \ac{Q--Q} testing for \ac{i.i.d.} Exp(1) interarrival times. (b) A qualitative autocorrelation test. The $U_k$ values are defined as $U_k~=~F(t^*_k-t^*_{k-1}) = 1-\e^{-(t^*_k-t^*_{k-1})}$.
}
  \label{final_pair}
\end{figure}

\subsubsection{Lewis test}

A statistical test with more power is the Lewis test as described by
\cite{kim2013}. Firstly, it relies on the fact that if $\{t^*_1,
t^*_2, \dots, t^*_N\}$ are arrival times for a unit rate Poisson
process then $\{t^*_1/t^*_N, t^*_2/t^*_N, \dots, t^*_{N-1}/t^*_N\}$
are distributed as the order statistics of a uniform $[0,1]$ random
sample. This observation is called conditional uniformity, and forms
the basis for a test itself. Lewis' test relies on applying Durbin's
modification (introduced in \cite{durbin1961} with a widely applicable
treatment by \cite{lewis1965}). \new

\subsubsection{Brownian motion approximation test} \label{bm_stuff}

An approximate test for Poissonity can be constructed by using the
Brownian motion approximation to the Poisson process. This is to say,
the observed times are transformed to be (approximately) Brownian
motion, and then known properties of Brownian motion sample paths can
be used to accept or reject the original sample. \new

The motivation for this line of enquiry comes from Algorithm 7.4.V of
\cite{daley2003a}, which is described as an ``approximate
Kolmogorov--Smirnov-type test''. Unfortunately, a typographical error
causes the algorithm (as printed) to produce incorrect answers for
various significance levels. An alternative test based on the Brownian
motion approximation is proposed here. \new

Say that $N(t)$ is a Poisson process of rate $T$. Define $M(t) =
(N(t)-tT)/\sqrt{T}$ for $t \in [0, 1]$. Donsker's invariance principle
implies that, as $T \to \infty$, $(M(t) : t \in [0,1])$ converges
in distribution to standard Brownian motion $(B(t) : t \in [0,1])$.
\fig{bm_approx} shows example realisations of $M(t)$ for various $T$
that, at least qualitatively, are reasonable approximations to
standard Brownian motion. \new

An alternative test is to utilise the first arcsine law for Brownian
motion, which states that the random time $M^* \in [0,1]$, given by
\[ M^* = \arg\max_{s \in [0, 1]} B(s)\,, \]
is arcsine distributed (that is, $M^* \sim \mathrm{Beta}(1/2, 1/2)$). \new
Therefore the test takes a sequence of arrivals observed over $[0, T]$ and:
\begin{enumerate}
\item transforms the arrivals to $\{t^*_1 / T, t^*_2 / T, \dots, t^*_k
/ T\}$ which should be a Poisson process with rate $T$ over $[0, 1]$,
\item constructs the Brownian motion approximation $M(t)$ as above,
finds the maximiser $M^*$, and
\item accepts the `unit-rate Poisson process' hypothesis if $M^*$ lies
within the $(\alpha/2, 1-\alpha/2)$ quantiles of the
$\mathrm{Beta}(1/2, 1/2)$ distribution; otherwise it is
rejected.
\end{enumerate}
As a final note, many other tests can be performed based on other
properties of Brownian motion. For example, the test could be based simply
on noting that $M(1) \sim \mathrm{N}(0,1)$, and thus accepts if $M(1) \in
[Z_{\alpha/2}, Z_{1-\alpha/2}]$ and rejects otherwise.

\begin{figure}[h]
	\centering
	\sidesubfloat[]{\includegraphics[width=0.45\textwidth]{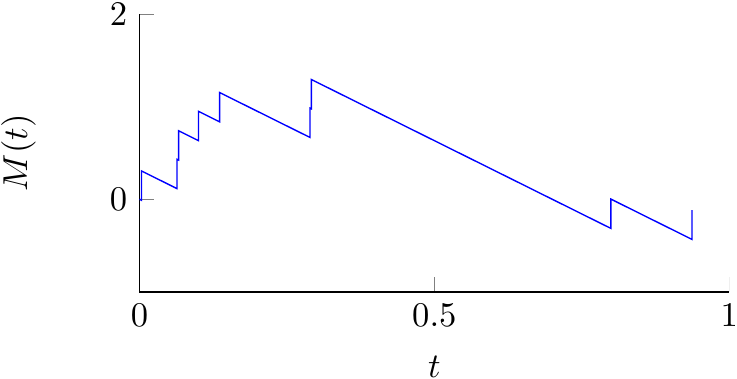}}~
	\sidesubfloat[]{\includegraphics[width=0.45\textwidth]{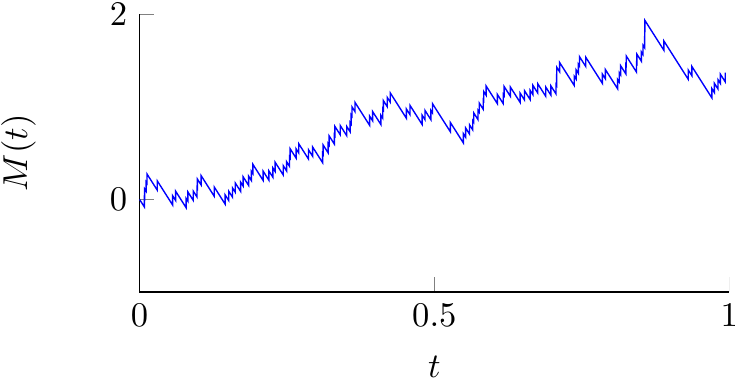}}\\
	\sidesubfloat[]{\includegraphics[width=0.45\textwidth]{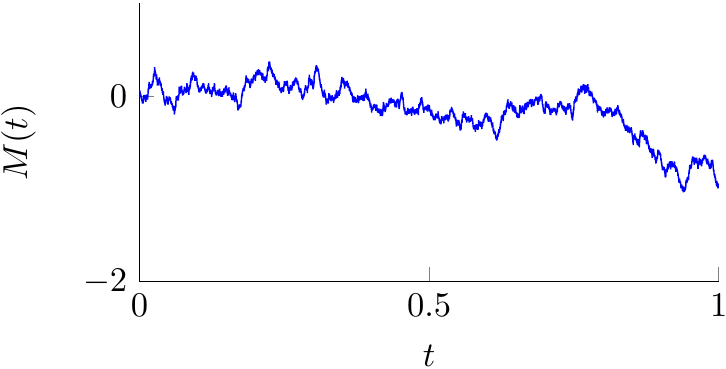}}~
	\sidesubfloat[]{\includegraphics[width=0.45\textwidth]{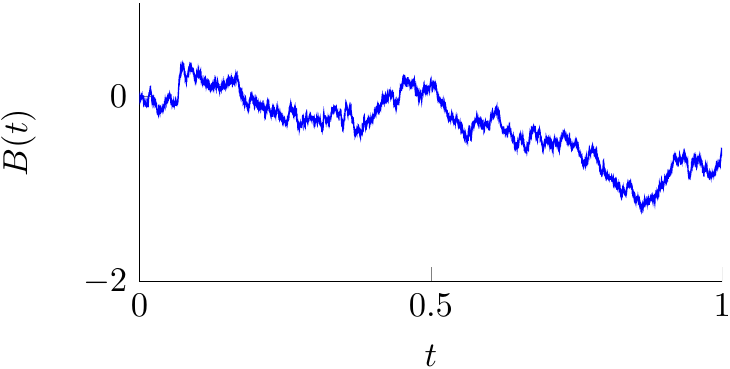}}\\
	\caption{Realisations of Poisson process approximations to Brownian motion. Plots (a)--(c) use the observed windows of $T=$10, 100, and 10,000, respectively. Plot (d) is a direct simulation of Brownian motion for comparison.}
	\label{bm_approx}
\end{figure}

\section{Simulation methods} \label{simulation}

Simulation is an increasingly indispensable tool in probability
modelling. Here we give details of three fundamental
approaches to producing realisations of \acp{HP}. \new

\subsection{Transformation methods}

For general point processes a simulation algorithm is suggested by the
converse of the random time change theorem (given in
\secn{transformation}). In essence, a unit rate Poisson process
$\{t^*_1, t^*_2, \dots\}$ is transformed by the inverse compensator
$\Lambda(\cdot)^{-1}$ into any general point process defined by that
compensator. The method, sometimes called the \emph{inverse
compensator method}, iteratively solves the equations
\[ t^*_1 = \int_0^{t_1} \cInt{s} \dif s, \quad t^*_{k+1} - t^*_k = \int_{t_k}^{t_{k+1}} \cInt{s} \dif s \]
for $\{t_1, t_2, \dots\}$, the desired point process (see \cite{giesecke2005} and Algorithm 7.4.III of \cite{daley2003a}). \new

For \acp{HP} the algorithm was first suggested by Ozaki
\cite{ozaki1979}, but did not state explicitly any relation to time
changes. It instead focused on \eq{generating_eq},
\[ \int_{t_k}^t \cInt{u} \dif u = - \log(1 - F^*(t)) \,, \]
which relates the conditional \ac{c.d.f.} of the next arrival to the previous history of arrivals $\{t_1, t_2, \dots, t_k\}$ and the specified $\cInt{t}$. This relation means the next arrival time $T_{k+1}$ can easily be generated by the inverse transform method, that is draw $U \sim \mathrm{Unif}[0,1]$ then $t_{k+1}$ is found by solving
\begin{equation} \label{to_solve}
	\int_{t_k}^{t_{k+1}} \cInt{u} \dif u = - \log(U) \,.
\end{equation}
For an exponentially decaying intensity the equation becomes
\[
  \log(U) + \bInt(t_{k+1} -t_k) - \frac{\alpha}{\beta}\Big(\sum_{i=1}^k \e^{\,\beta(t_{k-1}-t_i)} - \sum_{i=1}^k \e^{-\beta(t_k-t_i)} \Big) = 0 \,.
\]
Solving for $t_{k+1}$ can be achieved in linear time using the recursion of \eq{recursive}. However if a different excitation function is used then \eq{to_solve} must be solved numerically, for example using Newton's method \cite{ogata1981}, which entails a significant computational effort.

\subsection{Ogata's modified thinning algorithm}

\ac{HP} generation is a similar problem to inhomogeneous Poisson
process generation. The standard way to generate a inhomogeneous
Poisson process driven by intensity function $\lambda(\cdot)$
is via thinning. Formally the process is described by
\alg{pp_thinning} \cite{lewis1979}. The intuition is to generate a
`faster' homogeneous Poisson process, and remove points
probabilistically so that the remaining points satisfy the
time-varying intensity $\lambda(\cdot)$. The first process' rate $M$
cannot be less than $\lambda(\cdot)$ over $[0, T]$. \new

A similar approach can be used for the \ac{HP}, called \emph{Ogata's
modified thinning algorithm} \cite{ogata1981,liniger2009}. The
conditional intensity $\cInt{\cdot}$ does not have an \ac{a.s.}
asymptotic upper bound, however it is common for the intensity to be
non-increasing in periods without any arrivals. This implies that for
$t \in (T_i, T_{i+1}]$, $\cInt{t} \leq \cInt{T_i^+}$ (that is, the time just
after $T_i$, when that arrival has been registered). So the $M$ value
can be updated during each simulation. \alg{hawkes_thinning} describes
the process and \fig{thinning_ex} shows an example of each thinning
procedure. \new

\begin{algorithm}
\caption{Generate an inhomogeneous Poisson process by thinning.} \label{pp_thinning}
\begin{algorithmic}[1]
\Procedure{PoissonByThinning}{$T$, $\lambda(\cdot)$, $M$}
\State{\textbf{require}: $\lambda(\cdot) \leq M$ on $[0, T]$}
\State $P \gets []$, $t \gets 0$.
\While{$t < T$}
	\State $E \gets \mathrm{Exp}(M)$.
	\State $t \gets t + E$.
	\State $U \gets \mathrm{Unif}(0,M)$.
	\If {$t < T$ \textbf{and} $U \leq \lambda(t)$}
		\State $P \gets [P,\, t]$.
	\EndIf
\EndWhile
\State \Return $P$
\EndProcedure
\end{algorithmic}
\end{algorithm}

\begin{algorithm}
\caption{Generate a Hawkes process by thinning.} \label{hawkes_thinning}
\begin{algorithmic}[1]
\Procedure{HawkesByThinning}{$T$, $\cInt{\cdot}$}
\State \textbf{require}: $\cInt{\cdot}$ non-increasing in periods of no arrivals.
\State $\epsilon \gets 10^{-10}$ (some tiny value $> 0$).
\State $P \gets []$, $t \gets 0$.
\While{$t < T$}
	\State Find new upper bound:
  \State \hskip2em $M \gets \cInt{t+\epsilon}$.
	\State Generate next candidate point:
  \State \hskip2em $E \gets \mathrm{Exp}(M)$, $t \gets t + E$.
	\State Keep it with some probability:
  \State \hskip2em $U \gets \mathrm{Unif}(0,M)$.
	\If {$t < T$ \textbf{and} $U \leq \cInt{t}$}
		\State $P \gets [P,\, t]$.
	\EndIf
\EndWhile
\State \Return $P$
\EndProcedure
\end{algorithmic}
\end{algorithm}

\begin{figure}
  \centering
  \sidesubfloat[]{\includegraphics[width=0.8\textwidth]{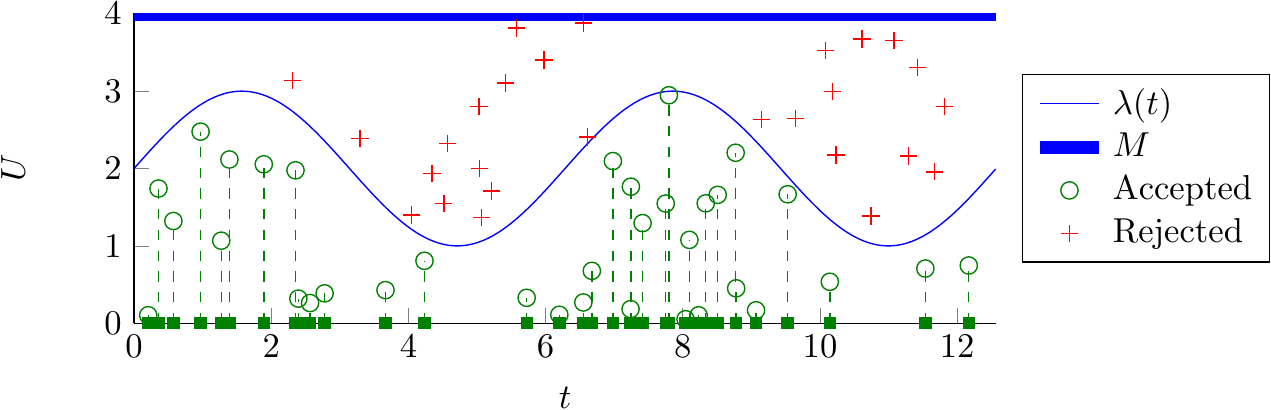}\label{pp_thinning_ex}} \\
  \sidesubfloat[]{\includegraphics[width=0.8\textwidth]{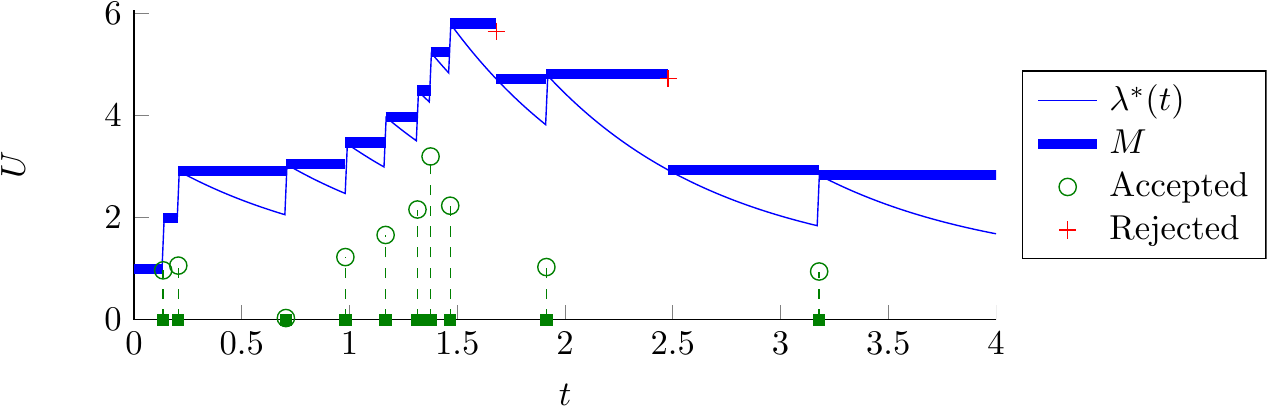}\label{hp_thinning_ex}\hspace{1pt}}
  \caption{Processes generated by thinning. (a) A Poisson process with intensity $\lambda(t) = 2 + \sin(t)$, bounded above by $M=4$. (b) A Hawkes process with $(\bInt, \alpha, \beta) = (1, 1, 1.1)$. Each $(t, U)$ point describes a suggested arrival at time $t$ whose $U$ value is given in \alg{pp_thinning} and \alg{hawkes_thinning}. Plus signs indicate rejected points, circles accepted, and green squares the resulting point processes.}
  \label{thinning_ex}
\end{figure}

\subsection{Superposition of Poisson processes}

The immigration--birth representation gives rise to a simple
simulation procedure: generate the immigrant arrivals, then generate
the descendants for each immigrant. \alg{hawkes_clusters} describes
the procedure in full, with \fig{hp_cluster} showing an example
realisation. \new

Immigrants form a homogeneous Poisson process of rate $\bInt$, so over
an interval $[0, T]$ the number of immigrants is Poi($\bInt T$)
distributed. Conditional on knowing that there are $k$ immigrants,
their arrival times $C_1, C_2, \dots, C_k$ are distributed as the
order statistics of \ac{i.i.d.} Unif[$0, T$] random variables. \new

Each immigrant's descendants form an inhomogeneous Poisson process.
The $i$th immigrant's descendants arrive with intensity
$\exite(t-C_i)$ for $t > C_i$. Denote $D_i$ to be the number of
descendants of immigrant $i$, then $\sE[D_i] = \int_0^\infty \exite(s)
\,\dif s = n$, and hence $D_i \iid \mathrm{Poi}(n)$. Say that the
descendants of the $i$th immigrant arrive at times $(C_i+E_1, C_i+E_2,
\dots, C_i+E_{D_i})$. Conditional on knowing $D_i$, the $E_j$ are
\ac{i.i.d.} random variables distributed with \ac{p.d.f.}
$\exite(\cdot)/n$. For exponentially decaying intensities, this
simplifies to $E_j \iid \mathrm{Exp}(\beta)$. \new

\begin{algorithm}[h]
\caption{Generate a Hawkes process by clusters.} \label{hawkes_clusters}
\begin{algorithmic}[1]
\Procedure{HawkesByClusters}{$T$, $\bInt$, $\alpha$, $\beta$}
\State $P \gets \{\}$.
\State Immigrants:
\State \hskip2em $k \gets \mathrm{Poi}(\bInt T)$
\State \hskip2em $C_1, C_2, \dots, C_k \getsIid \mathrm{Unif}(0, T)$.
\State Descendants:
\State \hskip2em $D_1, D_2, \dots, D_k \getsIid \mathrm{Poi}(\alpha/\beta)$.
\For{$i \gets 1$ \textbf{to} $k$}
	\If {$D_i > 0$}
		\State $E_1, E_2, \dots, E_{D_i} \getsIid \mathrm{Exp}(\beta)$.
		\State $P \gets P \cup \{C_i + E_1, \dots, C_i + E_{D_i}\}$.
	\EndIf
\EndFor
\State Remove descendants outside $[0, T]$:
\State \hskip2em $P \gets \{P_i : P_i \in P, P_i \leq T\}$.
\State Add in immigrants and sort:
\State \hskip2em $P \gets $ Sort($P \cup \{C_1, C_2, \dots, C_k\}$).
\State \Return $P$
\EndProcedure
\end{algorithmic}
\end{algorithm}

\begin{figure}
  \centering
  \sidesubfloat[]{\includegraphics[width=0.8\textwidth]{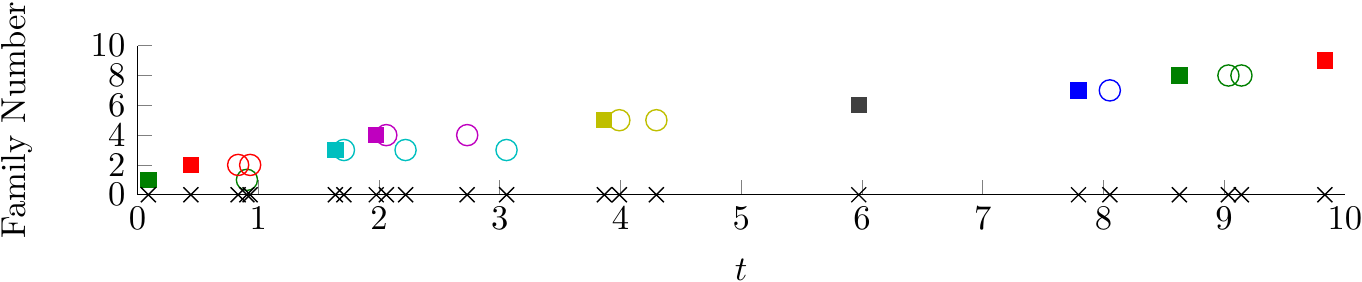}} \\
  \sidesubfloat[]{\includegraphics[width=0.8\textwidth]{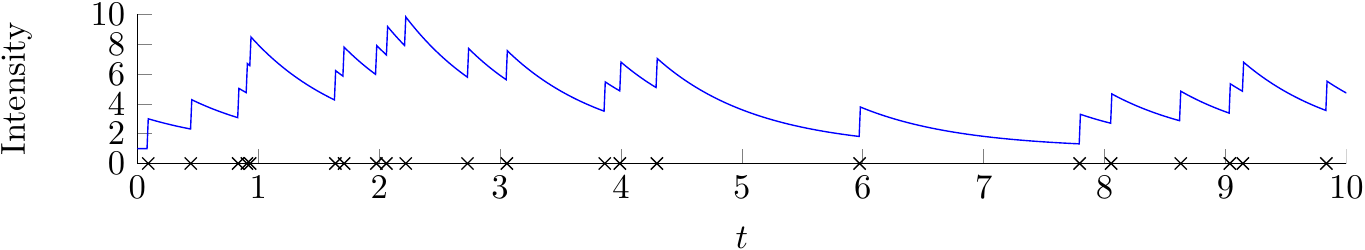}}
  \caption[A Hawkes process generated by clusters]{A Hawkes Poisson process generated by clusters. Plot (a) shows the points generated by the immigrant--birth representation; it can be seen as a sequence of vertically stacked `family trees'. The immigrant points are plotted as squares, following circles of the same height and color are its offspring. The intensity function, with $(\bInt,\alpha,\beta)=(1,2,1.2)$, is plotted in (b). The resulting Hawkes process arrivals are drawn as crosses on the axis.}
  \label{hp_cluster}
\end{figure}

\subsection{Other methods}

This section's contents are by no means a complete compilation of
simulation techniques available for \acp{HP}. Dassios and Zhao
\cite{dassios2013} and M{\o}ller and Rasmussen \cite{moller2005} give
alternatives to the methods listed above. Also not discussed is the
problem of simulating mutually-exciting \acp{HP}, however there are
many free software packages that provide this functionality.
\fig{mutual_example} shows an example realisation generated using the
R package `hawkes' (see also Roger D.\ Peng's related R package
`ptproc'). \new

\begin{figure}[H]
  \centering
  \sidesubfloat[]{\includegraphics[width=0.8\textwidth]{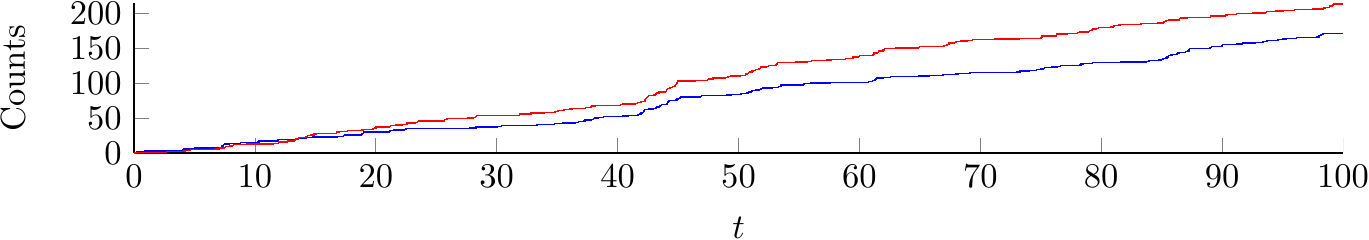}} \\
  \sidesubfloat[]{\includegraphics[width=0.8\textwidth]{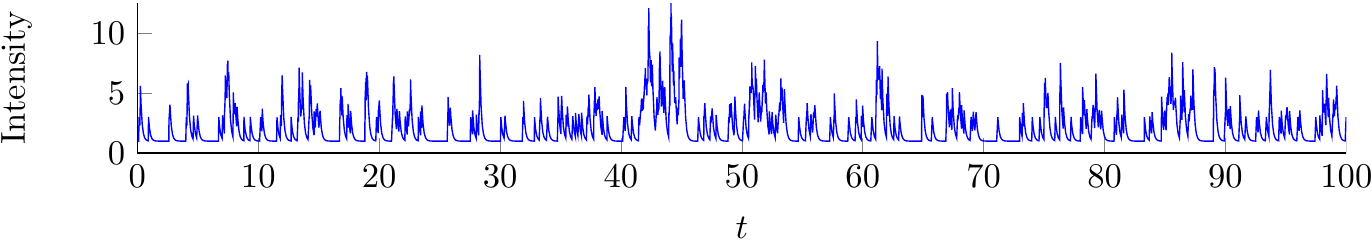}}
  \caption[An example of mutually exciting Hawkes processes]{A pair of mutually exciting Hawkes processes. (a) The two counting processes $N_1(t)$ and $N_2(t)$ with parameters: $\bInt_1=\bInt_2=1$, $\alpha_{1,1} = \alpha_{1,2} = \alpha_{2,1} = \alpha_{2,2} = 2$, $\beta_{1,1} = \beta_{1,2} = \beta_{2,1} = \beta_{2,2} = 8$. (b) The processes' realised intensitites (note that $\cInt[1]{t}=\cInt[2]{t}$ so only one is plotted).}
  \label{mutual_example}
\end{figure}

\section{Conclusion} \label{conclusion}

\acp{HP} are fundamentally fascinating models of reality. Many of the
standard probability models are Markovian and hence disregard the
history of the process. The \ac{HP} is structured around the premise
that the history matters, which partly explains why they appear in
such a broad range of applications. \new

If the exponentially decaying intensity can be utilised, then the
joint process $(N\fn, \cInt{\cdot})$ satisfies the Markov condition, and
both processes exhibit amazing analytical tractability. Explosion is
avoided by ensuring that $\alpha < \beta$. The covariance density is a
simple symmetric scaled exponential curve, and the power spectral
density is a shifted scaled Cauchy \ac{p.d.f.} The likelihood function
and the compensator are elegant, and efficient to calculate using
recursive structures. Exact simulation algorithms can generate this
type of \ac{HP} with optimal efficiency. Many aspects of the \ac{HP}
remain obtainable with any selection of excitation function; for
example, the random time change theorem completely solves the problem
of testing the goodness of a model's fit. \new

The use of \acp{HP} in finance appears itself to have been a
self-exciting process. A{\"\i}t-Sahalia et al.\ \cite{ait2010},
Filimonov and Sornette \cite{filimonov2012}, and Da Fonseca and
Zaatour \cite{dafonseca2014} formed the primary sources for the
financial part of \secn{lit_review}; these papers are surprisingly
recent (given the fact that the model was introduced in 1971) and are
representative of a current surge in \ac{HP} research. \new

\appendix

\section{Additional proof details} \label{math_appendix}

In this appendix, we collect additional detail elided from the proof of Theorem~\ref{spectral}.

\subsection{Supplementary to Theorem~\ref{spectral} (part one)} \label{part_i}

\begin{align*}
  R(\tau) &= \sE\left[ \frac{\dif N(t)}{\dif t} \left( \bInt + \int_{-\infty}^{t+\tau} \mu(t+\tau-s) \dif N(s) \right)\right] - \lm^2 \\
    &= \bInt \sE\left[\frac{\dif N(t)}{\dif t}\right] + \sE\left[\frac{\dif N(t)}{\dif t} \left(\int_{-\infty}^{t+\tau} \mu(t+\tau-s) \dif N(s) \right) \right] - \lm^2 \\
    &= \bInt \lm + \sE\left[\frac{\dif N(t)}{\dif t}\int_{-\infty}^{t+\tau} \mu(t+\tau-s) \dif N(s) \right] - \lm^2.
\end{align*}
Introduce a change of variable $v = s-t$ and multiply by $\frac{\dif v}{\dif v}$:
\[
  R(\tau) = \bInt \lm + \sE\left[\int_{-\infty}^{\tau} \mu(\tau-v) \frac{\dif N(t)}{\dif t} \frac{\dif N(t+v)}{\dif v} \dif v \right] - \lm^2
  = \bInt \lm + \int_{-\infty}^{\tau} \mu(\tau-v) \sE\left[ \frac{\dif N(t)}{\dif t} \frac{\dif N(t+v)}{\dif v}  \right] \dif v - \lm^2.
\]
The expectation is (a shifted) $R^{(c)}(v)$. Substitute that and \eq{complete} in:
\begin{align*}
  R(\tau) &= \bInt \lm + \int_{-\infty}^{\tau} \excite(\tau-v) \big(R^{(c)}(v) + \lm^2 \big)  \dif v - \lm^2 \\
    &= \bInt \lm + \int_{-\infty}^{\tau} \excite(\tau-v) \left(\lm \delta(v) + R(v)\right)  \dif v + \lm^2\int_{-\infty}^{\tau} \excite(\tau-v) \dif v - \lm^2 \\
    &= \bInt \lm + \lm \excite(\tau) + \int_{-\infty}^{\tau} \excite(\tau-v) R(v) \dif v + n\lm^2- \lm^2 \\
    &= \lm \excite(\tau) + \int_{-\infty}^{\tau} \excite(\tau-v) R(v) \dif v + \lm\,(\bInt -(1-n)\lm)\,.
\end{align*}
Using \eq{lm} yields
\[ \bInt -(1-n)\lm = \bInt -(1-n)\frac{\bInt}{1-n} = 0\,. \]
\[ \therefore\,\, R(\tau) = \lm \excite(\tau) + \int_{-\infty}^{\tau} \excite(\tau-v) R(v) \dif v \,. \]

\subsection{Supplementary to Theorem~\ref{spectral} (part two)} \label{part_ii}

Split the right-hand side of the equation into three functions $g_1, g_2$, and $g_3$:
\begin{equation} \label{another_to_trans}
  R(\tau) = \underbrace{\vphantom{\int_0^{\infty}} \lm \excite(\tau)}_{g_1(\tau)}+ \underbrace{\int_0^{\infty} \excite(\tau + v) R(v) \dif v}_{g_2(\tau)} + \underbrace{\int_0^{\tau} \excite(\tau - v) R(v) \dif v}_{g_3(\tau)} \,.
\end{equation}
Taking the Laplace transform of each term gives
\[
  \Laplace{g_1(\tau)}(s) = \int_0^s \e^{-s\tau} \lm \alpha \e^{-\beta \tau} \dif \tau= \frac{\alpha}{s+\beta} \lm\,,
\]
\begin{align*}
  \Laplace{g_2(\tau)}(s) &= \int_0^\infty \hspace{-5pt} \e^{-s\tau} \hspace{-4pt} \int_0^{\infty} \alpha \e^{-\beta(\tau + v)} R(v) \dif v \dif \tau \\
  &= \alpha \int_0^\infty \hspace{-5pt} \e^{-\beta  v} R(v) \int_0^{\infty} \e^{-\tau(s+\beta)} \dif \tau \dif v \\
  &= \frac{\alpha}{s + \beta} \int_0^\infty \e^{-\beta  v} R(v) \dif v \\
  &= \frac{\alpha}{s + \beta} \Laplace{R}(\beta) \,,
\end{align*}
and
\[
  \Laplace{g_3(\tau)}(s) = \Laplace{\excite(\tau)}(s) \Laplace{R(\tau)}(s)
  = \frac{\alpha}{s + \beta} \Laplace{R(\tau)}(s) \,.
\]
Therefore the Laplace transform of \eq{another_to_trans}
\begin{equation} \label{laplace}
    \Laplace{R(\tau)}(s) = \frac{\alpha}{s+\beta} \left( \lm  + \Laplace{R(\tau)}(\beta) + \Laplace{R(\tau)}(s) \right) \,.
\end{equation}
Substituting $s=\beta$ and rearranging gives that
\begin{equation} \label{beta}
    \Laplace{R(\tau)}(\beta) = \frac{\alpha \lm}{2(\beta - \alpha)} \,.
\end{equation}
So substituting the value of $\Laplace{R(\tau)}(\beta)$ into \eq{laplace} means
\[
  \Laplace{R(\tau)}(s) = \frac{\alpha}{s+\beta} \Big(\lm + \frac{\alpha \lm}{2(\beta - \alpha)} + \Laplace{R(\tau)}(s) \Big)
\]
\[ \Rightarrow \Laplace{R(\tau)}(s) = \frac{ \frac{\alpha}{s+\beta} \left( \lm  + \frac{\alpha \lm}{2(\beta - \alpha)} \right) }{1 - \frac{\alpha}{s+\beta}} = \frac{\alpha \lm ( 2\beta - \alpha)}{2(\beta - \alpha)(s + \beta - \alpha)} \,. \]

\bibliographystyle{spphys}
\bibliography{main}

\end{document}